\documentclass[11pt]{amsart}

\usepackage{dsfont}
\usepackage{amsmath, amsthm, amssymb, mathrsfs, eufrak}
\usepackage[abbrev,non-sorted-cites]{amsrefs}
\usepackage[all]{xy}
\usepackage{CJK}
\usepackage[abs]{overpic}
\usepackage{grffile}
\usepackage{graphicx}
\usepackage[usenames, dvipsnames]{color}
\newtheorem{thm}{Theorem}[section]
\newtheorem{cor}[thm]{Corollary}
\newtheorem{lem}[thm]{Lemma}
\newtheorem{ex}[thm]{Example}
\newtheorem{definition}[thm]{Definition}
\newtheorem{rmk}[thm]{Remark}

\newtheorem{prop}[thm]{Proposition}
\newtheorem{conj}[thm]{Conjecture}

\begin{document}

\title{Enumerative Geometry of Del Pezzo Surfaces}


\author[Y.-S. Lin]{Yu-Shen Lin}
\address{Department of Mathematics, Boston University, 665 Commonwealth Ave, Boston, MA 02215}
\email{yslin@bu.edu}
\thanks{Y.-S. L. is supported by Simons collaboration grant \# 635846 and NSF grant DMS2204109.}



\date{}

\dedicatory{}

\begin{abstract}
	 We prove an equivalence between the superpotential defined via tropical geometry and Lagrangian Floer theory for special Lagrangian torus fibres in del Pezzo surfaces constructed by Collins-Jacob-Lin \cite{CJL}. 
	We also include some explicit calculations for the projective plane, which confirm some folklore conjectures in this case.
\end{abstract}

\maketitle
  \section{Introduction}
Started from the celebrated work of Candelas-de la Ossa-Green-Parkes \cite{CDGP} computing the number of rational curves of a given degree in a generic quintic, mirror symmetry is believed to provide a new efficient way of computing enumerative invariants for Calabi-Yau manifolds. 
The mirror symmetry conjecture is further extended to Fano manifolds \cite{K2}. When $Y$ is a Fano manifold, its mirror is called the Landau-Ginzburg model $W:(\mathbb{C}^*)^n\rightarrow \mathbb{C}$, where the superpotential $W$ is some holomorphic function and $n=\mbox{dim}_{\mathbb{C}}Y$. The superpotential $W$ is expected to recover the enumerative geometry of $Y$. For instance, the quantum cohomology $QH(Y)$ is isomorphic to the Jacobian ring $Jac(W)$ \cite{B6}\cite{FOOO6} in the case of toric manifolds. In addition, the periods of $e^W$ with respect to suitable cycles recover the descendant Gromov-Witten invariants \cite{B7}\cite{G11}\cite{CCGGK} in the case of toric Fanos and thus known as the quantum periods. Due to the motivation from mirror symmetry, one would want to have a systematic way to write down the superpotential. In the toric Fano case, the superpotential is determined combinatorially by the toric polytope \cite{G10}\cite{HV}. It is further discovered by Cho-Oh \cite{CO} there are only positive Maslov index discs with boundaries on the moment map fibres and the superpotentials can be derived by the weighted count of Maslov index two discs with boundary on the moment map fibres and weighted by the symplectic area of holomorphic discs. 

In the pioneering work of Auroux \cite{A}, he studied the holomorphic discs with boundary on non-toric Lagrangian fibration constructed by Gross \cite{G2} outside of a singular anti-canonical divisor $D$ of some toric Fano manifold $Y$. In the presence of Maslov index zero discs, the superpotential is no longer constant with respect to the Lagrangian fibres. The loci in the base of Lagrangian fibration such that the corresponding fibres bound holomorphic discs of Maslov index zero form walls in the base. The superpotentials may jump when the Lagrangian boundary conditions vary across the walls, since Maslov index zero disc may glue together with a Maslov index two disc to a different Maslov index two discs on one side of the wall but not the other. This is further generalized to the case when the geometry has an equivariant $S^1$-action \cite{AAK}\cite{CLL}. 
When the anti-canonical divisor $D$ is smooth, which can be viewed as the deformations of the above results \cite{S9}. Auroux conjectured the existence of the special Lagrangians in the complement of the smooth anti-canonical divisor $X=Y\backslash D$ and had predictions on the SYZ mirror.  
The existence of the special Lagrangian was confirmed by Collins-Jacob-Lin \cite{CJL} in the case $Y$ is a del Pezzo surface. Furthermore, we will confirm the other half of the conjecture in \cite[Conjecture 7.3]{A}:
\begin{thm}(=Theorem \ref{115})
	There exists an open neighborhood $U$ such that the special Lagrangian torus fibre contained in $U$ bounds a unique holomorphic disc of Maslov index two disc contained in $U$. 
\end{thm}

Inspired by the Strominger-Yau-Zaslow conjecture \cite{SYZ} and the related heuristic in symplectic geometry, Gross-Hacking-Keel \cite{GHK} constructed the mirrors for the pair $(Y,D)$, where $Y$ is a projective surface with $D\in |-K_Y|$ wheel of rational curves. One first construct an affine manifold with a single singularity from the intersection matrix of $D$. This can be viewed as the limit of the base of the special Lagrangian fibration on $X=Y\backslash D$. Gross-Hacking-Keel further introduced the notion of broken lines and theta functions as the weighted count of broken lines. The mirror can then be realized as the spectrum of the algebra generated by the theta functions. The construction is extended to the case when $D$ is smooth by Carl-Pumperla-Siebert \cite{CPS}. 
In this paper, we adapt the definition of the broken lines to the setting that $B$ is the affine manifold from the special Lagrangian fibration constructed in \cite{CJL} and with the complex affine structure. The efforts are particularly made to prove that there is finitely many broken lines representing a fixed relative class. Therefore, one can define the weighted counts of the broken lines. Moreover, we provide the geometric meaning for the weight of broken lines that confirms the original symplectic heuristic.  
\begin{thm}(=Theorem \ref{163}) \label{main thm}
	The weighted counts of the broken lines coincide with the countings of the Maslov index two discs, serve as the coefficients of the superpotential in Lagrangian Floer theory.
\end{thm}

In the standard toric case, the naive SYZ mirror of the moment toric fibration only gives an bounded open subset of the Landau-Ginzburg mirror. Hori-Vafa explained that there is a renormalization procedure to obtain the LG mirror. Auroux predicted that the renormalization corresponds to scaling the symplectic form $kc_1(Y)$, $k\rightarrow \infty$. This is indeed hidden in the proof of Theorem \ref{main thm}: the special Lagrangian torus with respect to the Tian-Yau metric $\omega_{TY}$ in $X$ are not Lagrangian with respect to arbitrary symplectic form on $Y$. To make sense of the superpotential, we proved that there exists a sequence of  K\"ahler forms $\omega_i$ on $Y$ such that $[\omega_i]=k_ic_1(Y)$ with $k_i\rightarrow \infty$ and every special Lagrangian fibre with respect to the Tian-Yau metric in $X$ is Lagrangian with respect to $\omega_i$ for $i\gg 1$. Moreover, the coefficients of the superpotential is well-defined and stabilized. In other words, Theorem \ref{main thm} provides a rigorous mathematical realization of the renormalization procedure of Hori-Vafa. See \cite[Remark 5.11]{FOOO9} for another treatment of renormalization. 


A priori, the integral affine manifold used in the Gross-Siebert program is unclear to be the same as the one from the base of special Lagrangian fibration. In the case of $Y=\mathbb{P}^2$, Lau-Lee-Lin \cite{LLL} prove that the affine structure used in Carl-Pumperla-Siebert coincides with the complex affine structure of the special Lagrangian fibration constructed by Collins-Jacob-Lin. With the explicit affine structure, we provide some explicit calculation of the enumerative invariants and explains the relation with the previous works of Bousseau \cite{B5}, Vianna \cite{V3}, Graefnitz \cite{G8}. In particular, the interaction of the results in the current paper with the previous works confirm several folklore conjectures in this setting. 

The arrangement of the paper is as follows: In Section 2, we review the special Lagrangian fibration constructed in \cite{CJL} and its properties. We review the tropical geometry on the base of the special Lagrangian with the complex affine structure in Section 3. Then in Section 4 we establish the tropical/holomorphic correspondence for Maslov index zero discs with boundary on special Lagrangian fibres. Then we prove the equivalence of counting of Maslov index two discs with the same boundary conditions with the weighted count of broken lines in Section 5. This includes an explanation of the renormalization process of Hori-Vafa \cite{HV}. 
In Section 6, we speculate the relation of certain open Gromov-Witten invariants of Maslov index zero discs with relative Gromov-Witten invariant with maximal tangency. In Section 7, we provide some explicit calculation in the case of $\mathbb{P}^2$ and connect to other people's work.  

\section*{Acknowledgment} 
The author would like to thank S.T. Yau for the constant support and encouragement. The author would like to thank Man-Wai Cheung, Paul Hacking, Siu-Cheong Lau, Tsung-Ju Lee, Cheuk-Yu Mak for helpful discussions. The author particularly want to thank Pierrick Bousseau for the discussions back in 2016 at MSRI. The calculations in Example \ref{168} is done together with P. Bousseau at the time.

\section{Special Lagrangian Fibration on Del Pezzo Surfaces} \label{91}
Let $Y$ be a del Pezzo surface or a rational elliptic surface and $D=s^{-1}(0)$ be a smooth anti-canonical divisor, where $s\in H^0(Y,K_Y(D))$. By adjunction formula, $D$ is an elliptic curve and there exists a hermitian metric $h$ of $N_{D/Y}$ with Chern class $c_1(Y)\cdot D$.  Tian-Yau constructed a Ricci-flat metric $\omega_{TY}$ on the complement $X=Y\backslash D$ \cite{TY1}. Moreover, together with the results of Hein \cite{H}, one has the asymptotic of the Tian-Yau metric near $D$:
\begin{align}\label{asymptotics1}
	\omega_{TY}\sim i\partial\bar{\partial}(\log{|s|_h^2})^{\frac{3}{2}}.
\end{align}
Locally, we write $h=e^{-\phi}$ and $w$ a complex coordinate defining $D=\{w=0\}$. Then explicitly we have the following asymptotic:
\begin{align}\label{asmptotics2}
	\omega_{TY}\sim \frac{1}{2}(-\log{|w|^2e^{-\phi}})^{-\frac{1}{2}}\big( \frac{dw}{w}+\partial \phi \big)\wedge \big(\frac{d\bar{w}}{\bar{w}}+\bar{\partial}\phi\big)+(\log{|w|^2e^{-\phi}})^{\frac{1}{2}}\pi^*\omega_D,
\end{align} where $\omega_D$ is the unique Ricci-flat metric on $D$ with K\"ahler class $c_1(N_{D/Y})$. We will refer the reader to \cite[Section 4.1]{CJL} for explicit calculation and more details about the geometry of Calabi model.
The asymptotics of the Tian-Yau metric is important for the later studies of the holomorphic discs near $D$ and also the the affine structures.

Auroux \cite{A} conjectured the existence of special Lagrangian fibration on $X$ when $Y=\mathbb{P}^2$ or $Y$ is a rational elliptic surface. The conjecture is proved generally for del Pezzo surfaces by Collins-Jacob-Lin.
\begin{thm}\cite{CJL} \label{92}
	Let $Y$ be a del Pezzo surface or rational elliptic surface and $D\in |-K_Y|$ be a smooth anti-canonical divisor. Then there exists a special Lagrangian fibration on $X=Y\backslash D$ with a special Lagrangian section and the base is diffeomorphic to $\mathbb{R}^2$. Moreover, the monodromy at infinity is conjugate to $\begin{pmatrix}1 & d \\0 & 1\end{pmatrix}$.
\end{thm}
We will use $B$ to denote the base of the fibration and the fibration to be $\pi:X\rightarrow B$. 
\begin{thm} \cite{CJL} \label{res}
	Let $\check{X}$ be the same underlying manifold of $X$ with K\"ahler form $\check{\omega}$ and holomorphic volume form $\check{\Omega}$ given by 
	\begin{align}\label{HK rel}
		\check{\omega}&= \mbox{Re}\Omega,\notag \\
		\check{\Omega}&= \omega-i\mbox{Im}\Omega.
	\end{align}
	This can be compactified to a rational elliptic surface $\check{Y}$ by adding an $I_d$ fibre $\check{D}$, where $d=(-K_Y)^2$. 
\end{thm}	
We now sketch the idea of the proof of Theorem \ref{92} since that will help to understand the geometry needed for later sections. Let $Y_{\mathcal{C}}$ be a small neighborhood of zero section of the total space of the normal bundle $N_{D/Y}$ and $X_{\mathcal{C}}$ be the complement of the zero section in $Y_{\mathcal{C}}$. There exists a natural holomorphic volume form $\Omega_{\mathcal{C}}$ and Ricci-flat metric $\omega_{\mathcal{C}}$ invariant under the $S^1$-action along the fibres on $X_{\mathcal{C}}$. Choose a special Lagrangian $l$ on $D$ and it is straight-forward to check that the constant norm $S^1$-bundle over $l$ is a special Lagrangian torus in $X_{\mathcal{C}}$. In particular, a choice of special Lagrangian fibration in $D$ can be lifted to a special Lagrangian fibration on $X_{\mathcal{C}}$. These are the model special Lagrangian fibration near infinity. The difficulty is to prove the convergence of Lagrangian mean curvature flow under this setting with degenerate geometry. Then one can use hyperK\"ahler rotation and the theory of $J$-holomorphic curves to prove that the deformation of the limiting special Lagrangian tori from Lagrangian mean curvature flow sweep out $X$ and thus form a special Lagrangian fibration on $X$. 
\subsection{Affine Structures on the Base of SYZ Fibrations}   \label{affine structure}  
It is a standard fact by Hitchin \cite{H2} that 
the base $B$ of the special Lagrangian fibration naturally admits structures of affine manifold with singularities , i.e., the transition functions away from the discriminant locus fall in $GL(n,\mathbb{Z})\rtimes \mathbb{R}^n$. The base is where tropical curves/discs live and we will discuss later.

The affine structures are defined as follows:
assume that $X$ is a complex $n$-fold admitting a nowhere vanishing holomorphic $n$-form $\Omega$. 
We will denote $L_u$ the fibre over $u\in B$. 
Let $B_0$ be the complement of the discriminant of the fibration. 
Then $\cup_{u\in B_0}H_2(X,L_u;\mathbb{Z})$ forms a local system of lattices over $B_0$. By choosing relative classes $\gamma_1,\cdots,\gamma_n$ as flat local sections in an open neighborhood on $B_0$ such that their boundaries generate the first homology of fibres, then the functions $\mbox{Im}\int_{\gamma_i}\Omega$ give the integral affine coordinates on the neighborhood. This is known as the complex affine structure when $X\rightarrow B$ is a special Lagrangian fibration from the aspects of mirror symmetry. In particular, this gives a natural identification of $H^1(L_u,\mathbb{R})\cong T_uB_0$. Notice that this identification gives the opposite orientation of $T_uB_0$. Due to the natural pairing on $H_1(L_u,\mathbb{R})$ induces an isomorphism $H_1(L_u,\mathbb{R})\cong H^1(L_u,\mathbb{R})$, we will use the identifications freely later.      

The following observation reduces the complication of tropical geometry discussed in the next section. 
\begin{lem}\label{generic}
	For a generic choice of $(Y,D)$, the singular fibres of the special Lagrangian fibration of $X=Y\backslash D$ in Theorem \ref{92} are either simple nodal curves or cusp curves.
\end{lem}
\begin{proof}
	From Theorem \ref{res}, it suffices to prove that the corresponding rational elliptic surface $\check{Y}$  has only irreducible singular fibres from Kodaira's classification of singular fibres of minimal elliptic fibration. In particular, these components of reducible fibres have self-intersection $-2$. We will first briefly recall the Torelli theorem for rational elliptic surfaces \cite{M5}\cite{Hacking}. Let $\check{D}_i$ be the irreducible components of $\check{D}$ and 
	\begin{align*}
		Q=\{\alpha\in \mbox{Pic}(\check{Y})|\alpha\cdot\check{D}_i=0\}.
	\end{align*}
	Consider the long exact sequence of pairs with coefficients in $\mathbb{Z}$. 
	\begin{align*}
		0=H_3(Y)\rightarrow H_3(\check{Y},\check{X})\rightarrow H_2(\check{X})\rightarrow H_2(\check{Y})\rightarrow H_2(\check{Y},\check{X}).
	\end{align*} From Poincare duality $H_k(\check{Y},\check{X})\cong H^{4-k}(\check{D})$, one has the short exact sequence 
	\begin{align*}
		0\rightarrow H^1(\check{D})\rightarrow H_2(\check{X})\rightarrow Q\rightarrow 0. 
	\end{align*}  Denote the image of the generator of $H^1(\check{D})$ by $\delta$ and normalize $\int_{\delta}\check{\Omega}=1$. Then the period domain of rational elliptic surface $\mbox{Hom}(Q,\mathbb{C}^*)$ can be interpreted as the periods $\exp{\big(2\pi i\int\check{\Omega}\big)}:H_2(\check{X})\rightarrow \mathbb{C}^*$ \cite{L15}\cite{GHK2} (or the rational elliptic surface case \cite{B8}). Similarly, the periods of $\Omega$ integrating over cycles in $H_2(X)$ describe the moduli space of $(Y,D)$ . 
	Therefore, it suffices to prove that for the generic choice of $(Y,D)$, the period $\int_{\alpha} \check{\Omega}$ is non-zero on any $(-2)$-class $\alpha\in H_2(\check{X})$. Since $\check{\Omega}=\mbox{Im}\Omega+i\omega$ and $\omega$ is exact on $X$, it is equivalent to $\int_{\alpha}\mbox{Im}\Omega=0$ under the normalization $\int_{[L]}\Omega\in \mathbb{R}$, where $[L]\in H_2(X)$ is the fibre class. In other words, $\int_{\alpha}\Omega$ 
	and $\int_{[L]}\Omega$ have the same phase. The later is the zero locus of a non-zero harmonic function and thus analytic Zariski closed inside the moduli space of pairs of del Pezzo surface with a smooth anti-canonical divsior. 
\end{proof}

\subsection{Some Useful Consequence of the Asymptotics}
Here we collect some implications of the asymptotics of the Tian-Yau metric in (\ref{asymptotics1}).

The first lemma has its own interest in differential geometry: this interprets the Tian-Yau metric as the large volume limit of K\"ahler forms of classes along the multiple of $c_1(Y)$. 
\begin{lem} \label{6}
	There exist a sequence of increasing open subsets $\mathcal{U}_k$ of $Y\backslash D$ and K\"ahler forms $\omega_k$ of $Y$ such that 
	\begin{enumerate} 
		\item $\mathcal{U}_k\subseteq \mathcal{U}_{k+1}$ and $\cup_k \mathcal{U}_k=X$,
		\item $\omega_k|_{\mathcal{U}_k}=\omega_{TY}$ coincides with the Tian-Yau metric.
		\item $[\omega_k]$ is a multiple of $c_1(Y)\in H^2(Y,\mathbb{R})$. 
	\end{enumerate} In particular, the Tian-Yau metric can be viewed as of class "$\infty[c_1(Y)]$". More precisely, 
	\begin{align*}
		\lim_{i\rightarrow \infty}\omega_i|_{X}=\omega_{TY}. 
	\end{align*}     
	
\end{lem}
\begin{proof}
	Take any K\"ahler form $\omega_Y$ on $Y$ in the cohomology class $c_1(Y)$, then from the long exact sequence of relative pairs
	\begin{align} \label{Looijenga}
		H_2(X)\rightarrow H_2(Y)\rightarrow H_2(Y,X)\cong H^2(D),
	\end{align}
	we have $\omega_Y|_{H_2(X,\mathbb{Z})}=0$. Since the de Rham cohomology and singular cohomology of $X$ coincide, $\omega_Y$ is exact. From vanishing of Dolbeault cohomology of Stein manifolds, $\omega_Y=i\partial\bar{\partial}\phi$ for some real-valued function $\phi$ globally defined on $X$.
	In particular, $\phi$ is strictly plurisubharmonic. Since $\mathcal{O}_Y(D)$ is trivial on $X$, one can interpret $e^{-\phi}$ as a hermitian metric of the bundle. There is a unique nowhere vanishing section of $\mathcal{O}_Y(D)$ over $X$, up to a $\mathbb{C}^*$-scaling. By choosing the trivialization $\mathcal{O}_Y(D)|_X\cong X\times \mathbb{C}$ sending that section to $1$, one has the asymptotic $\phi\sim O(\log{|w|})$, where $D=\{w=0\}$ locally. In particular, on has $\lim_{x\rightarrow D}\phi=\infty$. 
	This implies that each level set $\phi^{-1}(c)$ is compact and thus $\phi$ is an exhaustive function. Now since $\phi$ is an exhaustive function, one can take $\mathcal{U}_k=\phi^{-1}((-\infty,k))$ relatively compact and $\cup_k \mathcal{U}_k=X$.
	
	Let $\chi_k:\mathbb{R}\rightarrow \mathbb{R}$ be a convex function such that $\chi_k(x)=0$ if $x\leq k-1$ and $\chi_k(x)=x$ if $x\geq k$. Then $\chi_k\circ \phi$ is again plurisubharmonic. Therefore, $\beta_k=i\partial\bar{\partial}(\chi_k\circ \phi)$ is semi-positive and coincides with $\omega_Y$ on $X\backslash \mathcal{U}_k$. In particular, $\beta_k$ can be extended over $D$ and defined on $Y$. Since $\beta_k$ is exact on $X$, one has $[\beta_k]$ is a multiple of $c_1(Y)$ from (\ref{Looijenga}). 
	
	Choose a cut-off function $\eta_k$ such that $\eta_k(x)=0$ for $x\in X\backslash U_{k+1}$ and $\eta_k(x)=1$ for $x\in \mathcal{U}_k$. Then set 
	\begin{align*}
		\omega_k= i\partial \bar{\partial}(\eta_k\phi_{TY})+t_k\beta_k,
	\end{align*} where $\phi_{TY}$ is the potential function of Tian-Yau metric. Then $\omega_k=\omega_{TY}$ on $U_{k-1}$ and $\omega_k=\omega_Y$ on $X\backslash U_{k+1}$. One can choose large enough $t_k$ to secure the positivity on $U_{k+1}\backslash U_{k-1}$ and the lemma is proved. From the above proof, We will take $\mathcal{U}_k$ be of the form $\mathcal{U}_k=\pi^{-1}(U_k)$ for some open subset $U_k\subset B$ such that ${U}_k\nearrow B$ and $\mathcal{U}_k\supseteq \{|w|>\epsilon_k\}$.
\end{proof}
The next lemma is helpful when study the behaviors of the Maslov index two discs. 
\begin{lem} \label{bound}
	Fix a K\"ahler form $\omega_Y$ on $Y$. The sequence of K\"ahler form $\omega_i$ can be taken such that there exists $C_1,C_2>$ independent of $i$ with
	\begin{align*}
		g_i \leq C_1\epsilon_i^{-20}g_Y \hspace{2mm}&\mbox{on $Y$, and} \\
		C_2 \epsilon_i^{-20}g_Y\leq g_i\hspace{2mm}&\mbox{on $\mathcal{U}_k$, $k\gg1$.}
	\end{align*} 
	Here $g_Y, g_i$ are the Riemannian metric associate to $\omega_Y,\omega_i$. Moreover, $[\omega_i]=k_ic_1(Y)$ with $k_i=O(\epsilon_i^{-20})$. 
\end{lem}
\begin{proof}
	Since $g_Y$ is defined on $Y$, there exists a constant $C_3,C_4$ such that on $U\cap X$ we have
	\begin{align} \label{bound'}
		C_3 (-\log{|w|})^{\frac{1}{2}}	g_{Y}   \leq  g_{TY}\leq C_4 (-\log{|w|})^{\frac{1}{2}}|w|^{-2} g_{Y}
	\end{align} on $\mathcal{U}_{k_0}^c$, for some fixed $k_0\gg 1$, from the asymptotic of the Tian-Yau metric. 
	Since $\mathcal{U}_k$ is relative compact, there exists a constant $C_5>0$ such that $C_5^{-1}g_{Y}\leq g_{TY}\leq  C_5 g_{Y}$ on $\mathcal{U}_{k_0}$. It suffices to take $t_k=|\epsilon_k|^{-20}$ in Lemma \ref{6}. The last statement follows from the construction in Lemma \ref{6}. 
\end{proof} 


%

\section{Tropical Geometry} \label{tropical geometry}
The modified SYZ conjecture \cite{KS4}\cite{GW} expects
the collapsing of the Calabi-Yau manifolds near large complex structure limit to affine manifolds with singularities. In particular, the holomorphic curves collapse to certain $1$-skeletons known as the tropical curves. It is proved in the case of toric manifolds, the tropical curves countings recover  the log Gromov-Witten invariants \cite{M2}\cite{NS}. 

Through out the paper, we will assume that $B$ is the base of the special Lagrangian fibration, the discriminant locus $\Delta\subseteq B$ are isolated points and $B_0=B\backslash \Delta$ equipped with the complex affine coordinates. Then there is a natural short exact sequence of lattices over $B_0$.
\begin{align*}
	0\rightarrow H_2(X)\rightarrow \Lambda_c=\bigcup_{u\in B_0}H_2(X,L_u)\rightarrow \Lambda_g=\bigcup_{u\in B_0}H_1(L_u)\rightarrow 0
\end{align*} The intersection pairing on $H_1(L_u)$ lifts naturally to $H_2(X,L_u)$ and $\Lambda_c$ is called the charge lattice.
We now recall the notion of scattering diagrams on $B$. We will assume that around each singularity of the affine structure the monodromy is conjugate to $\bigl(
\begin{smallmatrix}
	1 & 1\\
	0 & 1
\end{smallmatrix} \bigr)$ or $\bigl(\begin{smallmatrix}
	0 & 1\\
	-1 & 1
\end{smallmatrix} \bigr)$ which is the generic case due to Lemma \ref{generic}. For our purpose to compare to Floer theory, we will first introduce Novikov field and the modification of the language of scattering diagram for our setting. 

Recall that the Novikov field is
\begin{align*}
	\Lambda:=\{\sum_{i\in \mathbb{N}}c_iT^{\lambda_i}|\lambda_i\in \mathbb{R},\lim_{i\rightarrow \infty}\lambda_i=\infty, c_i\in \mathbb{C}\}.
\end{align*}  We denote its maximal ideal by $\Lambda_+$, the ring of units by $U_{\Lambda}=\mbox{val}^{-1}(0)$, and $\Lambda^*=\Lambda\backslash \{0\}$. There is a natural discrete valuation on $\Lambda^*$
\begin{align*}
	\mbox{val}: \Lambda^* \longrightarrow &\mathbb{R} \\
	\sum_{i\in \mathbb{N}}c_iT^{\lambda_i}\mapsto & \lambda_{i_0},
\end{align*}
where $i_0$ is the smallest $i$ with $\lambda_i\neq 0$. 
We formally extend the domain of $\mbox{val}$ to $\Lambda$ by setting $\mbox{val}(0)=\infty$.

\begin{definition} 
	A scattering diagram $\mathfrak{D}$ on $B$ is a set of $3$-tuples $\{(l_{\mathfrak{d}},\gamma_{\mathfrak{d}}, f_{\mathfrak{d}}) \}_{\mathfrak{d}\in I}$, where 
	\begin{enumerate}
		\item $l_{\mathfrak{d}}$ is an affine ray emanating from a point $o_{\mathfrak{d}}\in B$.
		\item $\gamma_{\mathfrak{d}}\in \Gamma(l_{\mathfrak{d}},\Lambda)$ flat section with $\gamma_{\mathfrak{d}}$ is positively proportional to the tangent of $l_{\mathfrak{d}}$. 
		\item For $u\in l_{\mathfrak{d}}$, the slab function $f_{\mathfrak{d}}(u)$ is a power series in $T^{\omega_{TY}(\gamma_{\mathfrak{d}})}z^{\gamma_{\mathfrak{d}}}$ with $1$ as its constant term.
	\end{enumerate} such that given any $\lambda>0$, there are only finitely many $\mathfrak{d}$ with $f_{\mathfrak{d}}\neq 0 (\mbox{ mod }\Lambda_+)$. We define the support of the scattering diagram 
	$     \mbox{Supp}(\mathfrak{D})=\bigcup_{\mathfrak{d}\in I} l_{\mathfrak{d}}$, and the singularity of the scattering diagram $\mbox{Sing}(\mathfrak{D})$ to be the union of the staring points of $l_{\mathfrak{d}}$. 
	
\end{definition}
It worth mentioned that the second assumption of a ray $l_{\mathfrak{d}}$ in the scattering diagram implies that $l_{\mathfrak{d}}$ is oriented such that $\omega_{TY}(\gamma_{\mathfrak{d}})$ is strictly increasing along $l_{\mathfrak{d}}$. It is proved by the author that there are five families of holomorphic discs with non-trivial local open Gromov-Witten invariants near a type $II$ singular fibre \cite{L12}. This motivates the following definition of the initial scattering diagram canonically associated to $B$ below. 
\begin{definition} \label{initiall}
	The initial scattering diagram $\mathfrak{D}_{in}=\{(l_{\mathfrak{d}},\gamma_{\mathfrak{d}},f_{\mathfrak{d}})\}$ consisting of 
	\begin{enumerate}
		\item $(l_{\pm},\gamma_{\pm},f_{\pm})$ for each singularity with monodromy conjugate to $\bigl(\begin{smallmatrix}
			1 & 1\\
			0 & 1
		\end{smallmatrix} \bigr)$ such that 
		\begin{itemize}
			\item $l_{\pm}$ are affine rays emanating from the singularity with tangents in the monodromy invariant direction.
			\item  $\gamma_{\pm}$ are the Lefschetz thimbles such that $\int_{\gamma_{\pm}}\omega_{TY}>0$ along $l_{\pm}$. 
			\item The slab functions are $f_{\pm}=1+T^{\omega_{TY}(\gamma_{\pm})}z^{ \gamma_{\pm}}$. 
		\end{itemize}
		\item $(l_i,\gamma_i,f_i)$, $i=1,\cdots, 5$ for each singularity with monodromy conjugate to $\bigl(\begin{smallmatrix}
			0 & 1\\
			-1 & 1
		\end{smallmatrix} \bigr)$ such that 
		\begin{itemize}
			\item After choosing a branch cut and  $\gamma'_1,\gamma'_2\in H_2(X,L_u)$ such that the counterclockwise monodromy is 
			\begin{align*}
				q\gamma'_1&\mapsto -\gamma'_2\\
				\gamma'_2&\mapsto \gamma'_1+\gamma'_2.
			\end{align*}
			\item $\gamma_1=-\gamma_1',\gamma_2=\gamma'_2, \gamma_3=\gamma'_1+\gamma'_2,\gamma_4=\gamma'_1,\gamma_5=-\gamma'_2$.
			\item $l_{i}$ is the affine line defined by $\gamma_i$. 
			\item The slab functions are  $f_i(u)=1+T^{\omega_{TY}(\gamma_i)}z^{\gamma_i}$.
		\end{itemize}
		
	\end{enumerate}
\end{definition}

Given a scattering diagram $\mathfrak{D}$ on $B$, $u\in B_0$ and $\lambda>0$. For a generic counterclockwise
$ \phi:S^1\rightarrow B$
in a small enough neighborhood of $u$ such that it intersects every ray $l_{\mathfrak{d}}$ starting from $u$ (or passing through $u$)
transversally and exactly once (or twice respectively) if $u\in l_{\mathfrak{d}}$
and $f_{\mathfrak{d}}(u)\not \equiv 0 (\mbox{ mod }T^{\lambda})$. Assume the
intersection order is $\mathfrak{d}_1,\cdots \mathfrak{d}_s$, then
we form an ordered product as follows :
\begin{align} \label{1011}
	\theta_{\phi,\mathfrak{D}}^{u,\lambda}=\mathcal{K}_{\mathfrak{d}_1}(u)^{\epsilon_1}\circ
	\dots \circ \mathcal{K}_{\mathfrak{d}_s}(u)^{\epsilon_s} \hspace{3mm} (\mbox{mod } T^{\lambda}),
\end{align} where each $\mathfrak{d}=\mathfrak{d}_i$ on right hand side of (\ref{1011}) are transformations of the form
\begin{align} \label{2016'}
	\mathcal{K}_{\mathfrak{d}}(u):\Lambda[[H_2(X,L_u;\mathbb{Z})]]&\rightarrow\Lambda[[H_2(X,L_u;\mathbb{Z})]] \notag\\
	z^{\gamma}&\mapsto z^{ \gamma}f_{\mathfrak{d}}(u)^{\langle \gamma,\gamma_{\mathfrak{d}} \rangle}.
\end{align} and $ \epsilon_i=\mbox{sgn} \langle\phi'(p_i), \gamma_{\mathfrak{d}}(p_i)\rangle$, for $p_i\in \mbox{Im}\phi\cap
\mathfrak{d}_i$ and with the identification $T_uB_0\cong H_1(L_u,\mathbb{R})$. Recall that the definition of the complex affine structure allows one to have the natural identification $\phi'(p_i)\in T_{p_i}B_0\cong H_1(L_{p_i},\mathbb{R})$. 
One then defines 
\begin{align*}
	\theta_{\mathfrak{D},u}=\lim_{\leftarrow} \theta^{u,\lambda}_{\phi,\mathfrak{D}},
\end{align*} where $\phi$ falls in a nested sequence of neighborhood of $u$ converging to $u$ and $\lambda\rightarrow \infty$. 

The following is the key observation of Kontsevich-Soibelman \cite{KS1} that one can reconstruct mirror via $\mathfrak{D}_{in}$ which is expected to capture the information of holomorphic discs with special Lagrangian boundary conditions. 
\begin{thm} \cite[Theorem 6]{KS1}\label{complete}
	Given the initial scattering diagram $\mathfrak{D}_{in}$, there exists a unique  complete scattering diagram $\mathfrak{D}$ such that $\theta_{\mathfrak{D},u}=id$ for every $u\in B_0$ and $l_{\mathfrak{d}}\neq l_{\mathfrak{d}'}$ if $\mathfrak{d}\neq \mathfrak{d}'$. 
\end{thm}
The complete scattering diagram $\mathfrak{D}$ is closely related to tropical discs on  $B$.
\begin{definition} \label{15}
	A tropical curve (with stop) on $B$ is a $3$-tuple $(\phi,T,w)$ where $T$ is a rooted connected tree (with a root $x$).  We denote the set of vertices and edges by $C_0(T)$ and $C_1(T)$ respectively. Then the weight function $w:C_1(T)\rightarrow \mathbb{N}$ and the continuous map $\phi:T\rightarrow B$ satisfy the following:
	\begin{enumerate}
		\item For any vertex $v\in C_0(T)$, the unique edge $e_v$ closest to the stop is called the outgoing edge of $v$ and $w_v:=w(e_v)$. 
		\item For each $e\in C_1(T)$, $\phi|_e$ is either an embedding of affine segment on $B_0$ or $\phi |_e$ is a constant map. In the later case, $e$ is associated with an integral primitive tangent vector at $\phi(e)$ (up to sign) if $\phi(e)\notin \Delta$. The edge adjacent to $x$ is not contracted by $\phi$. See (3), (4) for the case when $\phi(e)\in \Delta$. 
		\item For each $v\in C_0(T)$, $v\neq x$ and $\mbox{val}(v)=1$ if and only if $\phi(v)\in \Delta$. Moreover, 
		\begin{enumerate}
			\item If $\phi|_{e_v}$ is an embedding, then $\phi(e_v)$ falls in one of $l_{\gamma_i}$ in Definition \ref{initiall}.
			\item If $\phi(e_v)\in \Delta$ is contracted, then the integral primitive tangent vector $v_e$ associate to $e$ is in the monodromy invariant direction in $T_yB$, $y=Exp_{\phi(e_v)}(\epsilon v_e)$ for some small $\epsilon>0$. 
		\end{enumerate}
		
		\item For each $v\in C_0(T)$, $v\neq x$ and $\mbox{val}(v)=2$, we have $\phi(v)\in \Delta$. Moreover, 
		\begin{enumerate}
			\item the edges $e_v^+,e_v^-$ adjacent to $v$ are not contracted by $\phi$.
			\item $\phi(e_v^{\pm})$ is in the monodromy invariant direction.
			\item $w(e_v^+)=w(e_v^-)$.
		\end{enumerate}
		\item For each $v\in C_0(T)$, $\mbox{val}(v)\geq 2$, we have the following assumption: 
		
		(balancing condition) Each outgoing tangent at $u$ along the image of each edge adjacent to $v$ is rational with respect to the above integral structure on $T_{\phi(v)}B$. Denote the outgoing primitive tangent vectors by $v_i$ and the corresponding weight by $w_i$, then 
		\begin{align} \label{19}
			\sum_i w_i v_i=0.
		\end{align}
		\item For each unbounded edge $e\in C_1(T)$, the tangent of $\phi(e)$ is in the monodromy invariant direction. 
		\item The Maslov index of $(h,w,T)$ is defined to be 
		\begin{align*}
			\mu(h)=2 (\#\mbox{unbounded edges}).
		\end{align*}
		
	\end{enumerate}
\end{definition}
We will identify the two tropical discs $(\phi,T,w)$ and $(\phi',T',w')$ if there exists a homeomorphism $f:T\rightarrow T'$ such that $\phi'\circ f=\phi$ and $w'\circ f=w$. It worth mentioned that here we allow tropical discs/curves with contracted edges due to the presence of the singularity which is different from the usual definition. 

Before we define the tropical disc counting invariants, we need to associate each tropical disc $(\phi,T, w)$ with stop at $u\in B_0$ a relative class in $H_2(X,L_u)/H_2(X)$. Let $h(e)$ be the edge adjacent to $u$ and $u'$ is the image of another vertex. Extend the affine segment in the side of $u$ and one obtains an affine ray emanating from $u'$ which is a ray $l_{\mathfrak{d}}$ in the scattering diagram $\mathfrak{D}$. We then define
\begin{align*}
	[h]:=\gamma_{\mathfrak{d}}(u)\in H_2(X,L_u)
\end{align*} to be the relative class associate to the tropical disc $(h,T,w)$. 
Then for the purpose of comparing to open Gromov-Witten invariants, one naturally has the following definition:
\begin{definition}
	Given $u\in B_0$ and $\gamma\in H_2(X,L_u)$ primitive, define \\ $\tilde{\Omega}^{trop}(d\gamma;u)\in \mathbb{Q}$ and $l(\gamma;u)\in \mathbb{N}$ via the equation
	\begin{align*}
		\sum_{\mathfrak{d}:\gamma_{\mathfrak{d}}(u)=\gamma, l_{\mathfrak{d}}\ni u}\log{f_{\mathfrak{d}}(u)}=\sum_{d\geq 1} d\tilde{\Omega}^{trop}(d\gamma;u)(T^{\omega_{TY}(\gamma)}z^{\gamma})^d,
	\end{align*} and 
	\begin{align*}
		l(\gamma;u)=\max_{(h,T,w): [h]=\gamma}\#\{e\in C_1(T)|h(e)\mbox{ not contracted}\}.
	\end{align*}
	
\end{definition}
\begin{rmk} Although here we define the invariants
	$\tilde{\Omega}^{trop}(\gamma;u)$ through the scattering diagram, they coincide with the weighted counts of admissible tropical discs representing $\gamma\in H_2(X,L_u)$ (see Section 4 \cite{L8}). See also the refined version \cite{L6}\cite{M7} and its interpretation \cite{KS5} \cite{CM2}.
\end{rmk}      

\subsection{Broken Lines and Tropical Superpotential}
For the comparison of the counting of Maslov index two holomorphic discs later, we will review the definition of broken lines that introduced in \cite{GHK}. 
\begin{definition} Let $B_0$ be the base of the special Lagrangian fibration in Theorem \ref{92} equipped with the complex affine structure.
	A broken line is a continuous map 
	\begin{align*}
		\mathfrak{b}:(-\infty,0]\rightarrow B_0
	\end{align*} with stop $\mathfrak{b}(0)=u$ with the below properties:
	there exist
	\begin{align*}
		-\infty=t_0<t_1<\cdots <t_n=0
	\end{align*} such that $\mathfrak{b}|_{[t_{i-1},t_i]}$ is affine. For each $i=1,\cdots, n$, there is an associate flat section $c_iz^{m^{\mathfrak{b}}_i(t)}\in H_2(X,L_{\beta(t)})$ (after the identification of relative classes via parallel transport, we may omit the dependence of $t$ for simplicity when there is no confusion) such that 
	\begin{enumerate}
		\item For each $i$, $\partial (m^{\mathfrak{b}}_i)$ is positively proportional to $\beta'(t)$ under the identification $T_uB\cong H_1(L_u,\mathbb{R})$.
		\item $m^{\mathfrak{b}}_1=\beta_0$ and $c_1=1$. Here $\beta_0\in H_2(X,L_{\beta(0)})$ is defined as follows: recall that $\mathcal{U}_k, k\gg 1$ (in Lemma \ref{6}) is a neighborhood of infinity. One has $H_2(\mathcal{U}_k,L_{\beta(t)})\cong \mathbb{Z}, t\ll 0$ and there is a unique generator has positive pairing with the K\"ahler form $\omega$. Define $\beta_0$ (up to parallel transport) be the image $H_2(\mathcal{U}_k,L_{\beta(t)})\rightarrow H_2(X,L_{\beta(t)}), t\ll 0$. 
		\item $\mathfrak{b}(t_i)\in \mbox{Supp}(\mathfrak{D})\backslash \mbox{Sing}(\mathfrak{D})$. 
		\item If $\mathfrak{b}(t_i)\in \cap_i l_{\mathfrak{d}_i}$ (of $1$-dimensional intersection), then $c_{i+1}z^{m^{\mathfrak{b}}_{i+1}(t_{i+1})}$ is a term in 
		\begin{align}\label{broken line}
			\big(\prod_i\mathcal{K}_{\mathfrak{d}_i}(\mathfrak{b}(t_i)^{\epsilon_i}\big)(c_iz^{m^{\mathfrak{b}}_i(t_i)}),
		\end{align} where $\epsilon_i=\mbox{sgn}\langle m_i^{\mathfrak{b}}(t_i), \gamma_{\mathfrak{d}_i}\rangle$. 	
	\end{enumerate}
\end{definition}  
\begin{definition}
	Given a broken line $\mathfrak{b}$, we say that $n$ is the length of the broken line $\mathfrak{b}$ and it has the homology $[\mathfrak{b}]:=m^{\mathfrak{b}}_n(u)\in H_2(X,L_u)$ and weight $\mbox{Mono}(\mathfrak{b}):=c_n$. 
\end{definition}
The following is the simplest example of the broken line. 
\begin{ex}\label{initial broken line}
	Given $u$ close enough to infinity and there exists a unique affine line from $u$ to infinity with tangent in the monodromy invariant direction. This gives rise to a broken line $\mathfrak{b}:(-\infty,0]\rightarrow B_0$ such that its image is the above affine line and $n=1$ with the associate flat section $z^{\beta_0}$. The weight of this broken line is $1$. 
\end{ex}

\begin{rmk}
	We say $\mathfrak{b}$ is a degenerate broken line if it is a limit of broken lines $\mathfrak{b}_n$. Notice that the compatibility of the scattering diagram implies that $\mbox{Mono}(\mathfrak{b})$ is well-defined. 
\end{rmk}
We will denote $l_i$ to be the affine segment $\mathfrak{b}([t_{i-1},t_{i}])$. Recall that for each $u\in l_{\mathfrak{d}}$, there exists a tropical disc of Maslov index zero ending at $u$ with the edge adjacent to $u$ contained in $l_{\mathfrak{d}}$. Such tropical disc is not unique but represents a fixed relative class. We will denote $\phi_i$ the tropical disc of Maslov index zero associate to the bending point $\mathfrak{b}(t_i)$. 

\begin{rmk}
	The data for a broken line is equivalent to a Maslov index two tropical disc and the corresponding weighted counts coincide \cite{G9}. This indicates why later we will compare the weighted count of the broken lines and the counting of Maslov index two holomorphic discs. 
\end{rmk}
For the purpose of induction process later, we have the following definition: 
\begin{definition}
	We denote $l(\beta;u)$ be the largest possible length of broken lines representing $\beta\in H_2(X,L_u)$.
\end{definition}
A priori, it is not clear if $l(\beta;u)$ is finite. To prove that it is finite, we will first need to understand the behavior of affine rays near infinity. 

By \cite[Theorem 1.5]{CJL}, the special Lagrangian fibration near $D$ after hyperK\"ahler rotation via (\ref{HK rel}) is a germ of elliptic fibration over $U'_{\infty}$ and can be partially compactified by adding an $I_d$ fibre at infinity. After a suitable change of coordinate, it can be realized as $U'_{\infty}\times \mathbb{C}/\mathbb{Z}\langle 1, \frac{d}{2\pi i}\log{u}\rangle$ where $u$ is the holomorphic coordinate of the base with $u=0$ corresponds to the infinity, $x$ is that of the fibre after a suitable coordinate change and $d$ is the degree of the del Pezzo surface \cite{CJL2}. Moreover, the restriction of the holomorphic volume form with suitable normalization from (\ref{HK rel}) is given by $k(u)\frac{du}{u}\wedge dx$, $k(0)=1$ for some holomorphic function $k(u)$, since any holomorphic function is constant on the elliptic fibres. In general, $k(u)$ is not necessarily constant $1$. Notice that $h(u)=\frac{k(u)-1}{u}$ extends over infinity as a holomorphic function. Let $H(u)$ be the anti-derivative of $h(u)$ with vanishing constant term, then $\tilde{u}=e^{H(u)}u$. Then $d\tilde{u}=e^{H(u)}k(u)$ and $\tilde{u}$ defines a holomorphic coordinate near $u=0$ and $\Omega=\frac{d\tilde{u}}{\tilde{u}}\wedge dx$. The periods of the elliptic fibration is then $1, \frac{1}{2\pi i}\log{\tilde{u}}+f(\tilde{u})$, for some holomorphic function $f$ with $f(0)=0$. Since later we will not use the coordinate $u$ anymore, we will abuse the notation and still use $u$ to denote the new coordinates $\tilde{u}$ and set $u=e^{r+i\theta}$.  
With the above understanding, we will first study the behaviors of the affine rays near infinity. 

\begin{lem} \label{BPS ray behavior} There exists an open neighborhoods $U_{\infty}\subseteq V_{\infty}\subseteq  B$ of infinity such that the follow is true:
	Let $l$ be an affine ray in $U'_{\infty}$ labeled by $\gamma$ with respect to the above affine structure and starting from $u_0\in V_{\infty}$ such that $|u|$ is decreasing at $u_0$. Then $|u|$ is strictly decreasing and unbounded along the affine ray $l\cap U_{\infty}$. In particular, the affine ray $l$ converges to the infinity. 
\end{lem}
\begin{proof}
	Assume that $\partial \gamma$ is represented by
	the loop from $0$ to $m+n\cdot \frac{d}{2\pi i}\log{u}$. Then by straight-forward calculation, the affine ray emanating from $u_0$ satisfies the equation 
	\begin{align} \label{affine equation}
		m\big[\log{u}-\log{u_0}\big]+n \big[\frac{d}{2\pi i}\big(\frac{(\log{u})^2}{2}-\frac{(\log{u_0})^2}{2}\big)\big]+n\epsilon(u)\in \mathbb{R},
	\end{align} where $|\epsilon(u)|<C|u|$ with $C$ independent of $(m,n)$. 
	When $n=0$, i.e., $\partial \gamma$ is monodromy invariant, then solution to (\ref{affine equation}) is given by $u=u_0\mathbb{R}_{>0}$ and $u$ is decreasing along the affine ray.  
	Now for the case $n\neq 0$, the affine ray passing through $u_0$ satisfies the equation 
	\begin{align} \label{parabola}
		& m(\theta-\theta_0)+\frac{nd}{4\pi}(r^2-\theta^2-r_0^2+\theta_0^2)+n\mbox{Im}\epsilon(u)=0, \mbox{ or} \notag \\
		& r^2-(\theta-\frac{2m\pi}{nd})^2=r_0^2-(\frac{2m\pi}{nd}-\theta_0)^2   +\frac{4\pi}{d}\mbox{Im}\epsilon(u).      
	\end{align} 
	Let $C_{u_0,l}$ be the 
	parabola on the $(r,\theta)$-plane defined by \eqref{parabola} except the last term, which spirals into the infinity from both ends. The affine ray $l$ is a small $C^1$ deformation of $C_{u_0,l}$. It suffices to choose $V_{\infty}=\{u: \log{|u|}<-R\}$ for $R\gg 0$ and $U_{\infty}=\{u: \log{|u|}<-2R\}$.
	
\end{proof}
%
\begin{rmk}\label{position of vertex}
	From the proof of the lemma, for every affine ray $l$ starting from $u_0\in V_{\infty}$, there is an associate parabola $C$ as described in (\ref{parabola}) which is asymptotic to the affine ray $l$. Moreover, the vertex of $C$ is outside of $U_{\infty}$. 
\end{rmk}

%
	
	Next, we have the following consequence of Lemma \ref{BPS ray behavior}:
	\begin{lem} \label{MI0 sink} There exists an open neighborhood $W_{\infty}\subseteq U_{\infty}$ of infinity such that 
		if there is a tropical disc $\phi$ of Maslov index zero with end at $u_0\in W_{\infty}$, then $|u|$ is strictly decreasing and unbounded along the affine ray extending the edge adjacent to $u_0$. 
	\end{lem}
	\begin{proof}
		Since the singularities of the affine structures are outside of $U_{\infty}$, if a Maslov index zero disc $(\phi,T,w)$ has its stop $u_0\in V_{\infty}$ then some of the edges will have to enter $V_{\infty}$. 
		From Remark \ref{position of vertex}, any affine line $l$ entering $V_{\infty}$ the corresponding parabola will have vertex outside of $U_{\infty}$. If $l_i$ are such affine edges of $\phi$ intersecting at a point $u_0\in V_{\infty}$, then the corresponding parabola of the outgoing edge from $u_0$ has its vertex outside of $U_{\infty}$. Take $W_{\infty}=\{\log{|u|}>4R\}$. Let $l$ be an affine ray from $u_0$ and the associate vertex with vertex outside $U_{\infty}$, then $|u|$ is strictly decreasing along $l\cap W_{\infty}$. The lemma follows from induction on edges of $\phi(T)\cap V_{\infty}$. 
		
	\end{proof}

	Then we reach the key lemma to deduce the finiteness of $l(\beta;u)$. 
	\begin{lem}\label{MI2 asym} If the broken line 
		$\mathfrak{b}$ has a bending point $\mathfrak{b}(t_i)\in W_{\infty}$, then $|u|$ is strictly decreasing along $\mathfrak{b}([t_i,0])$.  	
		In particular, there exists a sequence of nested open neighborhood 
		$U_{\infty}^n\searrow \{\infty\}$  
		of $\infty\in B$ such that if a broken line $\mathfrak{b}$ has bending point in $U_{\infty}^n$, then its end falls in ${U}_{\infty}^n$. 
	\end{lem}
	\begin{proof}
		Assume that the first bending point $\mathfrak{b}(t_1)$ of the broken line $\mathfrak{b}$ is outside $W_{\infty}$. If there is a bending point $\mathfrak{b}(t_j)$ inside $W_{\infty}$, then there is an edge $l_i$ of $\mathfrak{b}$ entering $W_{\infty}$ with $i\leq j$. Then by Lemma \ref{BPS ray behavior}, one has $\mathfrak{b}(t_i)\in U_{\infty}$. From Lemma \ref{MI0 sink}, one has $|u|$ is strictly decreasing and converging to $u=0$ along the affine ray extending the edge of $\phi_i$ adjacent to $\mathfrak{b}(t_i)$. Then $|u|$ is strictly decreasing along $l_{i+1}$ by the same argument of Lemma \ref{MI0 sink}. By induction, one shows that $|u|$ is strictly decreasing and along $\mathfrak{b}([t_{i},0])$.
		%
		
		Now we are back to the situation that the first bending point of the broken line $u_0=\mathfrak{b}(t_1)=e^{r_0+i\theta_0}\in W_{\infty}$. There is an associate Maslov index zero tropical disc and tropical disc of Maslov index two described in Example \ref{initial broken line} with end at $\mathfrak{b}(t_1)$. Let $l$ be the affine ray containing edge adjacent to the stop. The former corresponds to a parabola $C_{u_0,l}$ with equation given by (\ref{parabola}) and the later is given by $\theta=\frac{2m\pi}{nd}$. Then the other affine segment $l'$ adjacent to $\mathfrak{b}(t_1)$ corresponds to the parabola $C_{u_0,l'}$ given by 
		\begin{align} \label{parabola2}
			r^2-(\theta-\frac{2(m+1)\pi}{nd})^2=r_0^2-(\frac{2(m+1)\pi}{nd}-\theta_0)^2 + \frac{4\pi}{d}\mbox{Im}\epsilon'(u),
		\end{align} for some $|\epsilon'(u)|<C|u|$. 
		From the above claim, $C_{u_0,l'}$ has its peak outside of $U_{\infty}$. To finish the proof of the lemma, we need to show that the $l$ is above the symmetry axis of $C_{u_0,l}$ if and only if $l'$ is above the symmetry axis of $C_{u_0,l'}$. Assume that $l$ is the symmetry axis of $C_{u_0,l}$. From the above claim, the $r$-coordinate of the vertex of $C_{u_0,l}$ is larger than $-2R$. Since $r_0<-4R$, we have 
		\begin{align*}
			|\theta_0-\frac{2m\pi}{nd}|^2= |r_0-r_0'|\cdot| r_0+r_0'|>3R\cdot 4R\gg 1,
		\end{align*} where $r_0'$ is the $r$-coordinate of the vertex of $C$,
		from the equation (\ref{parabola}). Since the symmetry axes of $C,C'$ has distance 
		\begin{align*}
			|\frac{2m\pi}{nd}-\frac{2(m+1)\pi}{nd}|<\pi,
		\end{align*} $l,l'$ always appear in the same side with respect to their symmetry axes. The $|u|$ is decreasing along $l$ implies that it is decreasing along $l'$ and the lemma if proved by induction and  Lemma \ref{MI0 sink}. 
	\end{proof}
	
	\begin{lem}\label{initial}
		For $u$ close enough to infinity, there exists a unique broken line representing $\beta_0$. 
	\end{lem}
	\begin{proof}
		It suffices to show that the broken line in Example \ref{initial broken line} is the only one broken line representing $\beta_0$. The lemma is then is a direct consequence of Lemma \ref{MI2 asym}. 
		
	\end{proof}
	\begin{lem} \label{affine distance}
		Let $A,B$ be two disjoint compact sets and $l$ is an affine line segment connecting $A,B$. Then there exists a constant $C>0$ independent of $l$ such that $\int_{[l]}\omega_{TY}>C$. 
	\end{lem}
	\begin{proof}
		First observe that $\int_{[l]}\omega_{TY}=|\int_{[l]}\check{\Omega}|$ from (\ref{HK rel}). After hyperK\"ahler rotation, we may identify the elliptic fibration locally as $U\times \mathbb{C}/\mathbb{Z}\oplus \mathbb{Z}\tau$. Under the identification, we have $\check{\Omega}=f(u)du\wedge dx$ with $f(u)\neq 0$. We may assume that $\check{\Omega}=du\wedge dx$ up to a suitable change of coordinate. 
		
		First assume that $l\subseteq U$. Let $u_0,u_1$ be the end points of $l$.
		Assume that $l$ is labeled by the $m+n\tau(u)\in H_1(L_u,\mathbb{Z})$.
		Choose a Riemannian metric on $B_0$. $l_u$ be the sub affine segment from $u_0$ to $u\in l$.  Recall that $\frac{d}{du}\int_{[l_u]}\check{\Omega}=\int_{m+n\tau(u)} \iota_v\check{\Omega}>0$, where $v$ is a lifting of a unit tangent vector of $l'$ from $u_0$ to $u$. Then 
		\begin{align*}
			\int_{[l]}\check{\Omega}&=\int_{u_0}^{u_1}\big( \int_{m+n\tau(u)}\iota_v\check{\Omega} \big)du \\ 
			& >\int_{u_0}^{u_1} |m+n\tau(u)| \min{\parallel \check{\Omega}\parallel} du \\
			& > C_1 (m^2+n^2)^{\frac{1}{2}} \mbox{length}(l)\\
			& > C_1 \cdot 1 \cdot \mbox{dist}(A,B)
		\end{align*} where $C_1>0$ is a constant depending on the metric, $\tau$, norm of $\check{\Omega}$ over $U$.   
		
	\end{proof}
	We will further need the following lemma to make sense of the weighted count of broken lines. Recall that $l(\beta;u)$ be the largest possible length of broken lines representing $\beta\in H_2(X,L_u)$.
	\begin{lem}\label{well-defined}
		Given $\beta\in H_2(X,L_u)$, then
		\begin{itemize}
			\item $l(\beta;u)<\infty$, and
			\item there are only finitely many broken lines representing $\beta$. 
		\end{itemize} 
	\end{lem}
	\begin{proof} Given a broken line $\mathfrak{b}$ with $[\mathfrak{b}]=\beta$.
		Assume that $u\notin U_{\infty}^n$ for some $n\in \mathbb{N}$, then all bending of $\mathfrak{b}$ is out side of $U_{\infty}^n$ by Lemma \ref{MI2 asym}. Choose $\omega_i$, $i\gg 1$ such that $\pi(\mathcal{U}_i)^c\subseteq  U_{\infty}^n$. Then $L_u$ is a Lagrangian with respect to $\omega_i$. In particular, the integral $\int_{\beta}\omega_i$ is well-defined and independent of the representative of $\beta$.

		
		Now for each singularity $p$ of $B$, there exist disjoint pairs of neighborhoods $U_p\subseteq U'_p$ disjoint from $\mathcal{U}_i^c$ such that any affine line segment connecting $\partial U_p,\partial U'_p$ has area greater than $C_2$ again from Lemma \ref{affine distance}. For each bending of $\mathfrak{b}$ in $\mathcal{U}_i$, the corresponding tropical disc of Maslov index zero either contains in $U_p$ or connecting $\partial U_p$ and $p$ for some singularity $p$. In the former case, either the tropical disc of Maslov index zero is in the above second case or the edge connecting the two bending point connects $\partial U_p,\partial U_{p'}$ for some singularities $p,p'$ and contributes more than $C_3$ by Lemma \ref{affine distance}. 
		To sum up, the number of bending of $\mathfrak{b}$ is bounded by $2\lambda/ \max{\{C_i\}}$. This implies the first part of the lemma that $l(\beta,u)$ is finite.

		
		Assume that there is a sequence of broken lines $\mathfrak{b}_n$ with $[\mathfrak{b}_n]=\beta$. Let $\mathcal{U}_i\in l$ be the last bending point of $\mathfrak{b}_i$. 
		Thus, there exists broken lines $\mathfrak{b}'_i$ and tropical disc $\phi_i$ of Maslov index zero with ends at $\mathcal{U}_i$ such that $[\mathfrak{b}'_i]+[\phi_i]=\beta$. By definition of the broken line and the correspondence theorem, there exists a holomorphic disc with boundary on $\mathcal{U}_i$ with relative class $[\phi_i]$. Notice that 	
		$\int_{[\phi_i]}\omega_i<\int_{\beta}\omega_i<\infty$ for a fixed $i\gg 1$. By compactness theorem, the above sequence of Maslov index zero discs has a convergent subsequence. In particular, $[\phi_i]$ converges after passing to a subsequence.
		
		Since Maslov index zero tropical discs are rigid and there are only finitely  many tropical tropical discs of Maslov index zero of a given homology class, we have $\mathcal{U}_i$ is constant (after passing to a subsequence). In particular, (after passing to a subsequence) $[\mathfrak{b}_n]$ has the same last edge and bending. By induction on the maximal number of bending points, the sequence stabilizes and there is only finitely many broken lines with class $\beta$. 
	\end{proof}
	Here the proof of Lemma \ref{well-defined} does largely rely on the correspondence theorem for discs of Maslov zero. However, the author would expect a purely tropical proof. 
	The direct consequence of Lemma \ref{well-defined} is the definition of the weighted count of degenerate broken lines.
	\begin{definition} \label{broken line counting}
		\begin{enumerate}
			\item Given a homology class $\beta\in H_2(X,L_u)$, define $n^{trop}(u)$ to be the weighted count of the degenerate broken lines 
			\begin{align*}
				n^{trop}_{\beta}(u)=\sum_{\mathfrak{b}:[\mathfrak{b}]=\beta} \mbox{Mono}(\mathfrak{b}),
			\end{align*} where the sum is over all possible degenerate broken lines representing $\beta$.
			\item Fix a symplectic form $\omega$ on $Y$ compatible with the standard complex structure $J$ and $L_u$ is a Lagrangian with respect to $\omega$. Denote the tropical superpotential $W^{trop}(u)$ to be 
			\begin{align*}
				W^{trop}(u):=\sum_{\beta\in H_2(X,L_u)} n^{trop}_{\beta}(u)T^{\omega(\beta)}z^{ \beta}.
			\end{align*}
		\end{enumerate}
		
	\end{definition}
	The proposition below follows from Lemma \ref{initial} directly. 
	\begin{prop}
		$n_{\beta_0}^{trop}(u)=1$ for $u$ close enough to the infinity. 
	\end{prop}

	The dependence of $n^{trop}_{\beta}(u)$ on $u\in B_0$ is discussed in \cite[Theorem 4.12]{G9}. We briefly recall the wall-crossing of $n^{trop}_{\beta}(u)$ below to make the paper self-contained: 
	Fix $u_0\in B_0$ and $\beta\in H_2(X,L_{u_0})$. Choose the K\"ahler form to be $\omega_i$, $i\gg 1$ such that all the possible bendings of broken lines representing $\beta$ is contained in $\pi(\mathcal{U}_i)$ by Lemma \ref{MI2 asym}. Then each possible associated tropical discs of Maslov index zero has area less than $\omega_i(\beta)$. Then by the compactness theorem and the correspondence theorem (Theorem \ref{116}), there exist only finitely many sets $\{l_{\mathfrak{d}_i}\}$ of rays in the scattering diagram $\mathfrak{D}$ satisfying the below properties: 
	\begin{enumerate}
		\item $l_{\mathfrak{d}_i}$ passes through $u_0$. 
		\item there exists a broken line end at $u_0$ representing $\beta'\in H_2(X,L_{u_0})$ and 
		\item $\beta=\beta'+\sum_i\gamma_{\mathfrak{d}_i}$.
	\end{enumerate}
	Assume that there is no set of rays $l_{\mathfrak{d}_i}$ with the property above. Then same is true for an open neighborhood of $u_0$. The counting of broken lines $n^{trop}_{\beta}(u)$ is then constant when $u$ is in a neighborhood of $u_0$ after we identify $\beta$ via parallel transport from the definition. 
	
	In general, $l_{\mathfrak{d}_i}$s divide a neighborhood of $u_0$ into finitely many chambers by Lemma \ref{well-defined}. For $u_1,u_2$ near $u_0$, one has $n^{trop}_{\beta}(u_1)=n^{trop}_{\beta}(u_2)$ if they are in the same chamber from the discussion above. Otherwise, one can reach $u_2$ from $u_1$ via passing through finitely many $l_{\mathfrak{d}_i}$s. Notice that if $\langle \gamma_{\mathfrak{d}},\gamma_{\mathfrak{d}'}\rangle=0$ and $l_{\mathfrak{d}},l_{\mathfrak{d}'}$ intersect at $u_0$, then $l_{\mathfrak{d}}=l_{\mathfrak{d}'}$ locally near $u_0$. In this case, it straight-forward to check that $\mathcal{K}_{\mathfrak{d}}(u), \mathcal{K}_{\mathfrak{d}'}(u)$ commute with each other and the order in the product doesn't matter. It suffices to understand how $n^{trop}_{\beta}(u)$ varies when $u$ move across a ray $l_{\mathfrak{d}_i}$. Notice that from the previous discussion, one may assume that $n^{trop}_{\beta'}(u)$ is constant in each chamber for all possible $\beta'$ associate to such sets.   
	Under the above assumption, the additional broken lines only appear in one side of $l_{\mathfrak{d}_i}$-where $\gamma_{\mathfrak{d}_i}$ is larger if one identify a neighborhood $u_0\in l_{\mathfrak{d}_i}$ with a neighborhood of $0\in H^1(L_{u_0},\mathbb{R})$ via the developing map. Assume that $(\gamma_{\mathfrak{d}_i},u_1)>(\gamma_{\mathfrak{d}_i},u_2)$, where $(,)$ is the natural pairing between homology and cohomology. 
	Then one has 
	\begin{align*}
		n^{trop}_{\beta}(u_1)-n^{trop}_{\beta}(u_2)= \sum_{\beta'}n_{\beta'}(u)c\big(f_{\mathfrak{d}_i}^{\langle \beta',\gamma_{\mathfrak{d}_i}\rangle}z^{\beta'},z^{\beta} \big),
	\end{align*} where $c(f,z^{\beta})$ denotes the coefficient of $z^{\beta}$ in $f$
	from (\ref{broken line}). Equivalently, one can derived $n^{trop}_{\beta}(u_2)$ from the coefficient of $z^{\beta}$ via the relation
	\begin{align*}
		W^{trop}(u_1)=\mathcal{K}_{\mathfrak{d}_i}W^{trop}(u_2).
	\end{align*}  

	\section{Floer Theory of Special Lagrangian Fibration on HyperKähler Surfaces}  This section is adapted from the earlier work in \cite{L4}\cite{L8}, which establish the correspondence theorem of holomorphic discs and tropical discs of Maslov index zero. 
	\subsection{$A_{\infty}$ Structures from Lagrangian Floer Theory}
	In this section, we have a brief review of some general results of $A_{\infty}$-structures from Lagrangian Floer theory developed by Fukaya-Oh-Ohta-Ono \cite{FOOO} and Fukaya \cite{F1}. 
	For the convergence issue, we first introduce the Novikov field $\Lambda$, 
	\begin{align*}
		\Lambda:=\mathbb{C}((T^{\mathbb{R}}))= \{ \sum_{i\geq 0}^{\infty} a_iT^{\lambda_i}| a_i\in \mathbb{C}, \lambda_i\nearrow \infty \}.
	\end{align*} Here $T$ is a formal variable. For the later purpose, we will also use the notation 
	\begin{align*}
		\Lambda_+=\{\sum_ia_iT^{\lambda_i}\in \Lambda| \lambda_i>0 \}. 
	\end{align*}  
	Now let $(X,\omega)$ be a compact symplectic manifold and $L\subseteq X$ be a relatively spin Lagrangian. Let $\mathcal{M}_{k,\beta}(X,L)$ (or $\mathcal{M}_{\beta}(X,L)$) denote the moduli space of stable discs in $X$ with boundaries on $L$ and relative class $\beta\in H_2(X,L)$ and $k$ (or without) boundary marked points. Then via the boundary relations of $\mathcal{M}_{k,\beta}(X,L)$ for all $k,\beta$, Fukaya-Oh-Ohta-Ono \cite{FOOO} and Fukaya \cite{F1} constructed an $A_{\infty}$-structure on $H^*(L,\Lambda)$. 
	\begin{thm}\cite[Corollary 12.1]{F1}
		There exits operators $m_{k,\beta}:H^*(L,\mathbb{R})^{\otimes k}\rightarrow H^*(L,\mathbb{R})$, $k\geq 0$ such that the operators 
		\begin{align*}
			m_k= \sum_{\beta}m_{k,\beta}T^{\omega(\beta)}: H^*(L,\Lambda)^{\otimes k}\rightarrow H^*(L,\Lambda)
		\end{align*} satisfy the $A_{\infty}$-relations 
		\begin{align*}
			\sum_{i,k_1+k_2=k+1}(-1)^*m_{k_1}(x_1,\cdots m_{k_2}(x_{i}\cdots x_{i+k_2-1}), x_{i+k_2}\cdots x_k)=0,
		\end{align*} for any $x_i\in H^{n_i}(L,\Lambda)$, where $*=\sum_i (n_i-1)$. 
	\end{thm}  	We will particular use the following version which is proved via the so-called de Rham model such that the divisor axiom \cite[Lemma 13.2]{F1} holds. 
	\subsection{Fukaya's Trick and Open Gromov-Witten Invariants}
	Now let $Y$ be a del Pezzo surface, $D$ be a smooth anti-canonical divisor and $X=Y\setminus D$.
	Although $X$ is non-compact, the moduli space of holomorphic discs $\mathcal{M}_{\gamma}(X,L_u)$ is still compact by Proposition 5.3 \cite{CJL} together with the usual Gromov compactness theorem (see also Theorem 4.10 \cite{G7}) and the Floer theory still applies.                
	The virtual dimension of $\mathcal{M}_{\gamma}(X,L_u)$ is negative one and thus $L_u$ does not bound any holomorphic discs for a generic almost complex structures. Indeed, we first recall the following well-known observation:
	\begin{lem} \label{11}
		If $L_{u(t)}$ bounds a holomorphic disc in $X$ of relative class $\gamma$ for every $t$, then $u(t)$ fall in an affine hypersurface of $B_0$.
	\end{lem}
	\begin{proof}
		Assume that $L_{u(t)}$ bounds a holomorphic disc in the relative class $\gamma$. Write $\gamma=a\gamma_1+b\gamma_2+\gamma_0$, where $\gamma_0\in H_2(X,\mathbb{Z})$. Then we have 
		\begin{align*}
			0&=\int_{\gamma_{u(t)}}\mbox{Im}\Omega\\
			&=\int_{a\gamma_{1,u(t)}+b\gamma_{2,u(t)}+\gamma_{0}}\mbox{Im}\Omega\\
			&= af_1(u(t))+bf_2(u(t))+ \int_{\gamma_0}\mbox{Im}\Omega. 
		\end{align*} The last term is a constant and thus the proposition follows.
	\end{proof}
	To detect the holomorphic discs bounded by the torus fibre above $l_{\gamma}$, we use the so-called Fukaya's trick which we now explain below: Fix a reference special Lagrangian $L_p$ and a path $\phi$ from $u_-$ to $u_+$ near $p$ on $B$. There exists a $1$-parameter family of diffeomorphisms $\phi_t:X\rightarrow X$ such that $\phi(L_p)=L_{\phi(t)}$. We will consider the $1$-parameter family of almost (actually integrable) complex structures $J_t=(\phi_t)^{-1}_*J$ on $X$ induced by the locally indecomposable $2$-form $\phi_t^*\Omega$. When $\phi$ is close enough to $p$, the almost complex structures are tamed by $\omega$ and the compactness still applies. Moreover, the moduli spaces of stable discs 
	\begin{align*}
		\mathcal{M}_{\gamma}(X,J_t,L_p)  \mbox{ and } \mathcal{M}_{(\phi_t)^{-1}_*\gamma}(X,J,L_{\phi(t)})
	\end{align*} are naturally identified and their corresponding Kuranishi structures agree. From \cite[Theorem 11.1]{F2}, the $1$-parameter family of tamed almost complex structures induces a pseudo-isotopy between two $A_{\infty}$ structures on $H^*(L_p)$. One comes from the holomorphic discs with boundaries on $L_{u_+}$ and another comes from the holomorphic discs with boundaries on $L_{u_-}$. In particular, the Maurer-Cartan moduli spaces associated to these two moduli spaces are isomorphic. 
	\begin{align*}
		F_{(\phi,p)}: H^1(L_{p},\Lambda_+)\cong H^1(L_{p},\Lambda_+).
	\end{align*}
	In general, $F_{(\phi,p)}$ may not be the identity but records the information from the holomorphic discs with boundaries on $L_{\phi(t)}$. 
	Recall that homotopic paths induce homotopic $A_{\infty}$ homomorphisms by \cite[Theorem 2.7]{T4}.
	Together with \cite[Lemma 4.3.15]{FOOO} which proves that the homotopic $A_{\infty}$-homomorphisms induced by the same map on the Maurer-Cartan spaces, we have 
	\begin{thm}\label{66} 
		Assume that $\phi$ is a loop homotopic to a constant loop on $B_0$ such that 
		\begin{enumerate}
			\item the homotopy is contained in a small enough open subset of $p$ and
			\item for all $u$ swept by the homotopy, $L_u$ does not bound any holomorphic discs of negative Maslov index.
		\end{enumerate}
		Then $F_{(\phi,p)}=\mbox{Id}$.  
	\end{thm}
	Due to the degree reason, only Maslov index zero discs will contribute to $F_{(\phi,p)}$. 
	To define the open Gromov-Witten invariants of Maslov index zero for the relative class $\gamma$, we may want to avoid certain real codimension one boundary of the moduli space. From Lemma \ref{11}, this only happens on an intersection of $l_{\gamma_1}$ and $l_{\gamma_2}$ with $\gamma=\gamma_1+\gamma_2$. By Gromov compactness theorem, locally there are only finitely many ways of such decomposition.       
	Assume that $\langle\gamma_1,\gamma_2\rangle\neq 0$, then such intersections are isolated points on $l_{\gamma}$ by maximum principle. Indeed, those intersection points satisfy        
	\begin{align} \label{74}
		\mbox{Arg}Z_{\gamma_1}(u)=\mbox{Arg}Z_{\gamma_2}(u),
	\end{align} where $Z_{\gamma}=\int_{\gamma}\omega-i\mbox{Im}\Omega$. Thanks to the fact that $Z$ is a holomorphic function, the locus defined by equation (\ref{74}) is a real codimension one submanifold on $l_{\gamma}$.
	Now let $u_0\in l_{\gamma}$ and $u\notin W_{\gamma}$, where 
	\begin{align*}
		W_{\gamma}=\bigcup_{\substack{\gamma_1+\gamma_2=\gamma \\ \langle \gamma_1,\gamma_2\rangle\neq 0}}W_{\gamma_1,\gamma_2}, \hspace{5mm}       W_{\gamma_1,\gamma_2}=\{u\in B|\mathcal{M}_{\gamma_i}(X,L_u)\neq \emptyset \}.
	\end{align*}        
	Let $u_{\pm}$ be in a small neighborhood of $u$ such that 
	\begin{align} \label{order}
		\mbox{Arg}Z_{\gamma}(u_+)>\mbox{Arg}Z_{\gamma}(u)>\mbox{Arg}Z_{\gamma}(u_-).
	\end{align} From the above discussion, there exists an isomorphism 
	\begin{align*}
		F_{(\phi,u)}:H^1(L_u,\Lambda_+)\cong H^1(L_u,\Lambda_+).
	\end{align*}
	Let $e_1,e_2$ be an integral basis of $H_1(L_{u},\mathbb{Z})$ and write $b=x_1e_1+x_2e_2\in H^1(L_{u},\Lambda_+)$. Motivated by the SYZ mirror symmetry, it is natural to consider $z_k=\exp{x_k}$ and the induced algebra homomorphism \footnote{It is actually an isomorphism \cite{T4}.}
	\begin{align*}
		\Lambda[[H_1(L_{u})]] \rightarrow \Lambda[[H_1(L_{u})]] \\
		F_{(\phi,u)}:z_k\mapsto \exp{\big(F_{(\phi,u)}(b)_k\big)},
	\end{align*} where $F_{(\phi,u)}(b)=F_{(\phi,u)}(b)_1e_1+F_{(\phi,u)}(b)_2e_2$ and $z^{\partial \gamma}:=z_1^{\langle \partial\gamma,e_1\rangle} z_2^{\langle \partial \gamma,e_2\rangle}$.
	
	There is no affine line $l_{\gamma}=l_{\gamma'}$ unless $\gamma$ is a multiple of $\gamma'$. We would choose as the sequence $\lambda_i\rightarrow \infty$, pairs of points $u_+^i,u_-^i$ and paths $\phi_i$ connecting $u_{\pm}^i$ as above such that 
	\begin{enumerate}
		\item $\lim_{i\rightarrow \infty}u^i_{\pm}=u$.
		\item $|Z_{\gamma}|$ is increasing along $\phi_i$.
		\item There exists no $\gamma'\in H_2(X,L_{\phi_i(t)})$ such that $|Z_{\gamma}(u)|<\lambda_i$ and there exists a holomorphic disc with boundary on $L_{\phi_i(t)}$ in the relative class $\gamma'$.
	\end{enumerate}
	Then by  \cite[Theorem 6.13]{L8}, the transformation $F_{(\phi_i,u)}$ is of the form 
	\begin{align*}
		F_{(\phi_i,u)}:z^{\partial \gamma'}\mapsto z^{\partial \gamma'}{f^i_{\gamma}(u)}^{\langle \gamma',\gamma\rangle} \mbox{ (mod $T^{\lambda_i}$)},
	\end{align*} where $\langle\gamma',\gamma\rangle$ is the intersection pairing of the corresponding boundary classes in $H_1(L_{u},\mathbb{Z})$ and $f^i_{\gamma}(u)\in 1+\Lambda_+[z^{\partial \gamma}]$. From Theorem \ref{66}, we have $f_{\gamma}(u):=\lim_{i\rightarrow \infty}f^i_{\gamma}(u)$ exists. Now we will define the open Gromov-Witten invariants as follows: 
	\begin{definition}\label{openGW}
		Given $\gamma\in H_2(X,L_u)$ and $u\notin W'_{\gamma}$. Then the open Gromov-Witten invariants $\tilde{\Omega}(\gamma;u)$ are defined by  
		\begin{align} \label{69}
			\log{f_{\gamma}(u)}=\sum_{d\geq 1}d\tilde{\Omega}(d\gamma;u)(T^{Z_{\gamma}(u)}z^{\partial\gamma})^d.
		\end{align} and we define the algebra isomorphism 
		\begin{align} \label{2016}
			\mathcal{K}_{\gamma}(u):\Lambda[[H_1(L_u)]]\rightarrow \Lambda[[H_1(L_u)]] \notag \\
			z^{\partial\gamma'}\mapsto z^{\partial \gamma'}f_{\gamma}(u)^{\langle \gamma',\gamma\rangle},
		\end{align} which is of the same form of (\ref{2016'}). 
	\end{definition}
	\begin{rmk} Here we defined the open Gromov-Witten invariants via the wall-crossing properties but
		we refer the readers to \cite[Theorem 6.29]{L8} for the comparison with the reduced counting of Maslov index zero discs. 
	\end{rmk}

	Motivated by the Gopakumar-Vafa conjecture \cite{GV}\cite{IP}, one can formulate the follow open analogue of the Gopakumar-Vafa conjecture\footnote{See also the corresponding integrality of log Gromov-Witten invariants for log Calabi-Yau surfaces \cite{CKGT}}: 
	\begin{conj} \label{12}
		There exist integers $\Omega(\gamma;u)$ such that 
		\begin{align} \label{13}
			\tilde{\Omega}(d\gamma;u)=-\sum_{k|d}c(\gamma)^d\frac{\Omega(\frac{d}{k}\gamma;u)}{k^2},
		\end{align} where $c=\pm 1$ is a quadratic refinement. Here we use the convention that $\Omega(\frac{d}{k};u)=0$ if $\gamma$ is not divisible by $k$. 
	\end{conj}
	Via M\"obius transformation, one can rewrite (\ref{13}) and the conjecture becomes
	\begin{align} \label{multiple cover}
		\Omega(d\gamma;u)=-\sum_{k|d}c(\gamma;u)^{\frac{d}{k}}\mu(k)\frac{\tilde{\Omega}(\frac{d}{k}\gamma;u)}{k^2}\in \mathbb{Z},
	\end{align} where $\mu(k)$ is the M\"obius function.

	\subsection{Correspondence Theorem}
	
	It is a folklore statement that near a focus-focus singularity there are two parameter families of holomorphic discs. Each of them is regular in a $1$-parameter family. It is proved for the case when locally the fibration and the almost complex structures are modeled by $\mathbb{C}^2$ \cite{FOOO5} or the Ooguri-Vafa space \cite{C}. 
	
	\begin{lem} \label{initial discs}
		Let $X\rightarrow U$ be a germ of elliptic fibration with a holomorphic volume form $\Omega$. Then 
		\begin{align*}
			\tilde{\Omega}(\gamma;u)=\begin{cases} \frac{(-1)^{d-1}}{d^2} & \mbox{if } \gamma= d\gamma_e  \\  0 & \mbox{otherwise} . \end{cases}
		\end{align*}
		Here $u\in U$ and $\gamma_e$ is the class of Lefschetz thimble with $\int_{\gamma_e}\omega_{TY}>0$. 
	\end{lem} 
	\begin{proof}
		Given two K\"ahler forms $\omega_1$ and $\omega_{2}$ and set $\omega_t=t\omega_1+(1-t)\omega_{2}$. Then $d\omega_t=0$ and $\omega_t(\cdot, J\cdot)>0$, where $J$ is the complex structure of the elliptic fibration. In particular, $\omega_t$ are non-degenerate and thus symplectic forms. We will recall the following modification of the hyperK\"ahler rotation \cite{HL}\cite{G4}:
		Given a K\"ahler form $\omega$ on $X$, one has $\omega\wedge \Omega=0$ for type reason. There exists a positive function $h$ such that $(h\omega)^2=(\mbox{Im}\Omega)^2$. Set $\omega'=h\omega$ and $\Omega_{\theta}=\omega'-i\mbox{Im}(e^{-i\theta}\Omega)$, then $\Omega_{\theta}^2=0$. Thus, $\Omega_{\theta}$ is decomposable and defines an almost complex structure $J_{\theta}$, which is integrable if and only if $\omega$ is Ricci-flat. Let $\omega_{\theta}=-\mbox{Re}(e^{-i\theta}\Omega)$ is a symplectic form on $X$ and it is straight-forward to check pointwisely that $J_{\theta}$ is a tamed almost complex structure. Without loss of generality, we may assume that $\Omega_{\theta}$ is positive on the fibres by exchanging $\omega_{\theta}\leftrightarrow -\omega_{\theta}$ and $J_{\theta}\leftrightarrow -J_{\theta}$.  The original fibration is "special Lagrangian" in the sense that the new symplectic form $\omega_{\theta}$ and pseudo-holomorphic volume form $\Omega_{\theta}$ restricted to zero on the fibres. 
		
		Thus, $\omega_t$ induces the interpolation of $S^1$-families of almost complex structures induced from $\omega_1,\omega_2$. Then similar argument in \cite[Corollary 4.23]{L8} shows that the invariant does not depends on the $t$. By choosing $\omega_1$ to be the Ricci-flat metric on the rational elliptic surface and $\omega_2$ to be the Ooguri-Vafa metric, then the lemma follows from \cite[Theorem 4.44]{L8}. The proof is motivated from the split attractor flows of counting of black holes \cite{DM}.
	\end{proof}
	Now we are ready to have the tropical/holomorphic correspondence of discs for the log Calabi-Yau surfaces with the proof slightly adapted from \cite{L8}. 
	\begin{thm} \label{116}
		Let $u\in B_0$ and $\gamma\in H_2(X,L_u)$ such that $\tilde{\Omega}(u,\gamma)$ is well-defined, then $\tilde{\Omega}^{trop}(\gamma;u)$ is well-defined and 
		\begin{align*}
			\tilde{\Omega}(\gamma;u)=\tilde{\Omega}^{trop}(\gamma;u). 
		\end{align*}
	\end{thm}
	\begin{proof}
		We will prove by induction on $l(\gamma;u)$. Assume that $l(\gamma;u)=1$, then there exists a unique tropical disc $(h,T,w)$ with end on $u$ and $[h]=\gamma$. Moreover, $h(T)$ is an affine line segment between a singularity and $u$ described in Definition \ref{initiall}. The assumption $l(\gamma;u)=1$ implies that $\tilde{\Omega}^{trop}(\gamma;u)=\frac{(-1)^{d-1}}{d^2}$, where $d$ is the divisibility of $\gamma$. In these cases, the theorem reduces to Lemma \ref{initial discs} and \cite[Theorem 4.11]{L12} when $u$ is close enough to a singularity. Now we will have the idea from split attractor flows 
		\begin{lem}\label{split attractor flow}
			Assume that $\tilde{\Omega}(\gamma;u)\neq 0$, then there exists a tropical disc ending at $u$.   
		\end{lem}		
		\begin{proof}
			Let $l$ be the affine ray emanating from $u$ defined by $\gamma$ such that $\int_{\gamma}\omega_{TY}$ is decreasing along $l$. First observe that there exists $\lambda>0$ such that if $\int_{\gamma}\omega_{TY}<\lambda$, then $\gamma$ is one of the relative classes (or their multiple) described in Definition \ref{initiall} by the monotonicity of holomorphic discs. In particular, the lemma holds when $\int_{\gamma}\omega_{TY}<\lambda$. 
			
			Assume that $\tilde{\Omega}(\gamma;u)$ is invariant when $u$ moves along $l$ and $\gamma$ is the parallel transport. Then $\int_{\gamma}\omega_{TY}$ decreases to zero at some point and $\gamma$ is one of those discs (or the multiple of them) described in Definition \ref{initiall}. Otherwise, notice that those jumps happen discretely by Gromov compactness theorem. Thus one may assume that $\tilde{\Omega}(\gamma;u)$ first jumps at $u'$.
			Then by Theorem \ref{66}, there exist $\gamma_i$ such that
			\begin{enumerate}
				\item $\sum_i \gamma_i=\gamma$,
				\item  $\tilde{\Omega}(\gamma_i;u')\neq 0$,
				\item  $\langle \gamma_{i_1},\gamma_{i_2}\rangle \neq 0$, and 
				\item $\int_{\gamma_i,u'}\omega< \int_{\gamma,u}\omega_{TY}$.
			\end{enumerate} By induction on the real number $\int_{\gamma,u}\omega_{TY}$, we may assume that there exist tropical discs representing $\gamma_i$ and end at $u'$. The gluing of these tropical discs and the union of the part of $l$ from $u'$ to $u$ is a tropical disc with end at $u$, where the balancing condition at $u'$ follows from $\sum_i \gamma_i=\gamma$.

		\end{proof}
		Similar to the above proof of Lemma \ref{split attractor flow}, there exists a tropical disc with end at $u$ and relative class $\gamma$ with more than two vertices unless $\tilde{\Omega}(\gamma;u)=\frac{(-1)^{d-1}}{d^2}$ and $\gamma=d\gamma_e$. In particular, this implies that $l(\gamma;u)>1$. Therefore, we have $\tilde{\Omega}(\gamma;u)=\frac{(-1)^{d-1}}{d^2}=\tilde{\Omega}^{trop}(\gamma;u)$ if $l(\gamma;u)=1$. 
		
		Assume that $\tilde{\Omega}(\gamma;u)=\tilde{\Omega}^{trop}(\gamma;u)$ is true for $l(\gamma;u)\leq k$. Let $(u,\gamma\in H_2(X,L_u))$ be a pair with $l(\gamma;u)=k+1$. 
		Again let $l$ be the affine ray emanating from $u$ defined by $\gamma$ such that $\int_{\gamma}\omega_{TY}$ is decreasing along $l$ and assume that $u'\in l$ is the first place where $\tilde{\Omega}(\gamma)$ jumps. Apply Lemma \ref{66} to a small loop around $u'$ then the theorem follows from the Kontsevich-Soibelman lemma \cite[Theorem 6]{KS1} and the induction hypothesis. 
	\end{proof}
	
	\begin{rmk}
		One can generalize the definition of initial scattering diagram to the case when there are singular fibres of the type  $II, III, IV$ according to the local open Gromov-Witten invariants computed in \cite{L12} and the tropical/holomorphic correspondence still holds. 
	\end{rmk}      
	\begin{rmk}
		From Perrson's list of singular configurations of rational elliptic surfaces \cite{P7}, a rational elliptic surface cannot contain one of a $I_n^*$, $IV^*$, $III^*$, $II^*$ together with an $I_n$ singular fibre for $n\geq 5$. In particular, the tropical/holomorphic correspondence holds when $Y$ is a del Pezzo surface of degree $\geq 5$. 
	\end{rmk}

	\section{Superpotential Functions and Tropical/Holomorphic Correspondence}

	\subsection{Definition of the Floer theoretic Superpotential}
	Before studying the superpotential of the del Pezzo surface $Y$, we will first construct a sequence of K\"ahler forms which are compatible with the special Lagrangian fibration constructed in \cite{CJL}. 
	
	\begin{rmk}
		This can be viewed as the mathematical realization of the Hori-Vafa renormalization. We will refer the readers to in \cite[Conjecture 4.4]{A} and the discussion below. 
	\end{rmk}

	First observe that $L_u$ are special Lagrangians with respect to $\omega_{TY}$ and $\Omega$. In particular, $L_u$ are totally real with respect to the standard complex structure $J$ and the compactness theorem holds \cite{FZ}.   
	From Lemma \ref{6}, the torus $L_u$ is Lagrangian with respect to $\omega_i$ for $i\gg 1$. Therefore, it make sense to consider the superpotential of $L_u$. 
	Given a relative class $\beta\in H_2(X,L_u)$ of Maslov index two. Assume that $L_u$ does not bound any Maslov index zero holomorphic discs, then the moduli space $\mathcal{M}_{1,\beta}(X,L_u)$ is compact without boundary. Define the open Gromov-Witten invariant of class $\beta$ by 
	\begin{align*}
		n_{\beta}(u):=m_{0,\beta},
	\end{align*} where $m_{0,\beta}$ is part of the $A_{\infty}$-operators of the $A_{\infty}$-structure on $H^*(L_u,\Lambda)$. 
	Intuitively, it counts the number of Maslov index two discs with boundary on $L_u$ and in relative homology class $\beta$ passing through a generic point on $L_u$. Intuitively, $n_{\beta}$ counts the number of holomorphic discs of class $\beta$ with the boundary passing through a generic point in $L_u$.
	Notice that the moduli space $\mathcal{M}_{1,\beta}(L_u)$ is independent of the choice of the symplectic form $\omega_i$ and thus so is the open Gromov-Witten invariants $n_{\beta}(u)$. On the other hand, fix a symplectic form $\omega$ such that $L_u$ is a Lagrangian and let $\{m_k\}_{k\geq 0}$ be the $A_{\infty}$ structure on $H^*(L_u,\Lambda)$ constructed by Fukaya-Oh-Ohta-Ono \cite{FOOO}. Take $b=x_1e_1+x_2e_2\in H^1(L_u,\Lambda_+)$ with $e_i$ basis in $H^1(L_u,\mathbb{Z})$. Recall that $L_u$ bounds no Maslov index zero holomorphic discs for generic $u\in B_0$. In this case, one has
	\begin{align}\label{m0}
		m(e^b):&=\sum_{k\geq 0}m_k(b,\cdots, b) \notag\\
		&=\sum_{\beta:\mu(\beta)=2}\exp{\langle \partial\beta,b\rangle} m_{0,\beta}T^{\omega(\beta)} \notag \\
		&=\sum_{\beta:\mu(\beta)=2} m_{0,\beta}z_1^{\langle \partial  \beta,e_1\rangle}z_2^{\langle \partial \beta,e_2\rangle} T^{\omega(\beta)} \notag \\
		&=\sum_{\beta:\mu(\beta)=2} n_{\beta}(u)T^{\omega(\beta)}z^{\partial\beta}.
	\end{align} Here the second equality is from the divisor axiom (see \cite[Lemma 13.2]{F1}).
	The compactness theorem guarantees that the right hand side of (\ref{m0}) is well-defined as a convergent in $T$-adic sense. 
	Recall that $m(e^b)$ is a multiple of $1_{L_u}$ by the degree reason, where $1_{L_u}$ is the unit of the $A_{\infty}$-algebra $H^*(L_u,\Lambda)$. Thus,
	we will define the superpotential $W(u)\in  \Lambda_+[[H_1(L_u,\mathbb{Z})]]$ by 
	\begin{align}
		m(e^b)=W(u)1_{L_u}.
	\end{align} 
	to emphasize its dependence on $u\in B_0$. 
	
	Given $u\in B_0$ and assume that $L_{u'}$ Lagrangian with respect to $\omega$ for $u'\in U$, where $U$ is a small enough contractible neighborhood of $u$. It is natural to compare $n_{\beta}(u)$ and $n_{\beta}(u')$. We may assume that at least one of $L_u$ or $L_{u'}$ bound holomorphic discs in the relative class $\beta$, otherwise $n_{\beta}(u)=n_{\beta}(u')=0$. Then take $\lambda>\omega(\beta)>0$ over $U$. Now fixed $u'\in U$ and a path $\phi$ from $u$ to $u'$ contained in $U$ and apply the Fukaya's trick and we have an $A_{\infty}$-homomorphism $F_{(\phi,u)}$. Then 
	\begin{align} \label{111}
		m(e^b)\equiv m(e^{F_{(\phi,u)}(b)})   (\mbox{ mod }T^{\lambda}), 
	\end{align} for any Maurer-Cartan element $b\in H^1(L_u,\Lambda_+)$ and $\lambda>0$ by \cite[Lemma 4.2.23]{FOOO}.
	There are only finitely many relative classes of holomorphic discs with boundaries on $L_{u'}$, $u'\in U$. In particular, $\phi$ intersects finitely many $l_{\gamma}$s. It suffices to understand the case when $\phi$ intersects exactly one $l_{\gamma}$ transversally and  (\ref{order}) holds. Then from (\ref{111}), we have  	
	\begin{align*}
		W(u')1_{L_u}&=m(e^{F_{(\phi,u)}(b)})\\
		&=\sum_{\beta:\mu(\beta)=2} n_{\beta}(u) \exp{\langle \partial \beta, F_{(\phi,u)}(b) \rangle } T^{\omega(\beta)}\\ 
		&=\sum_{\beta:\mu(\beta)=2} n_{\beta}(u) \exp{\langle \partial \beta, F_{(\phi,u)}(b)_1e_1+F_{\phi,u}(b)_2e_2 \rangle } T^{\omega(\beta)}\\ 
		&=\sum_{\beta:\mu(\beta)=2} n_{\beta}(u) \exp{\big(F_{(\phi,u)}(b)_1\big)}^{\langle \partial \beta, e_1\rangle}\exp{\big(F_{(\phi,u)}(b)_2\big)}^{\langle \partial \beta, e_2\rangle}  T^{\omega(\beta)}\\   
		&=\sum_{\beta:\mu(\beta)=2} n_{\beta}(u) F_{\phi,u}(z_1)^{\langle \partial\beta,e_1\rangle} F_{\phi,u}(z_2)^{\langle \partial\beta,e_2\rangle}  T^{\omega(\beta)} \\  
		&=\sum_{\beta:\mu(\beta)=2} n_{\beta}(u)F_{(\phi,u)}(z_1^{\langle \beta,e_1\rangle }z_2^{\langle \partial \beta,e_2\rangle}) T^{\omega(\beta)} \\ &=\sum_{\beta:\mu(\beta)=2}n_{\beta}(u)F_{(\phi,u)}(z^{\partial\beta})T^{\omega(\beta)}	= F_{(\phi,u)} W(u) 1_{L_u}.          
	\end{align*} Here the second equality is coming from the similar calculation in \eqref{m0}, the fifth equality comes the definition of $F_{(\phi,u)}$, the sixth equality comes from the fact that $F_{(\phi,u)}$ is a homomorphism and the last equality comes from the definition of $z^{\partial \beta}$.
	In particular, we end up with the wall-crossing formula of the superpotential with \eqref{2016}:
	\begin{lem} \label{WCF}
		$W(u')=K_{\gamma}W(u)$.
	\end{lem}
	
	In general, the torus fibre $L_u$ may bound holomorphic discs of Maslov index zero and the associate $A_{\infty}$ structure is only defined up to pseudo-isotopy \cite{F2}. Assume that $\{m_k'\}_{k\geq 0}$ is another $A_{\infty}$ structure due to different choices in the construction and $\{f_k\}_{k\geq 0}$ is the induced $A_{\infty}$-homomorphism. Let $W'(u)=m'(e^b)$ and one has $W(u)=KW'(u)$, where $K$ is a product of transformations of the form (\ref{2016}) and $K_{\gamma}$ appears in the product only if $L_u$ bounds a holomorphic disc of relative $\gamma$. Notice that the coefficient of $T^{\omega(\beta)}z^{\partial\beta}$ in $W(u),W'(u)$ is the same unless $L_u$ bound holomorphic discs $\beta', \gamma_i$, where $\beta$ is of Maslov index two and $\gamma_i$ of Maslov index zero such that $\beta=\beta'+\sum_i\gamma_i$. To sum up, $n_{\beta}(u)$ is the coefficient of $T^{\omega(\beta)}z^{\partial\beta}$ in $W(u)$ if $\mathcal{M}_{1,\beta}(L_u)$ admits no real codimension one boundary.

	The following lemma is the symplectic analogue of Lemma \ref{affine distance}.
	\begin{lem} \label{94}
		Given $C_1>0$, there exists a compact set $K$ such that if $L_u$ is a special Lagrangian torus in Theorem \ref{92} contained in $X\backslash K$ and bound a holomorphic discs of Maslov index zero in relative class $\gamma\in H_2(X,L)$ with $\tilde{\Omega}(\gamma;u)\neq 0 $ then $\omega_{TY}(\gamma)>C_1$. 
	\end{lem}
	\begin{proof}
		This is a direct consequence of monotonicity of holomorphic curves \cite{CJL}\cite{G7}. Notice that the geometry is not bounded but only degenerate mildly. 
	\end{proof}
	Fix a K\"ahler form $\omega_Y$ on $Y$ (with respect to the standard complex structure $J$). Let $\mathcal{U}_i=\{|w|\leq \epsilon_i\}^c$ near $D$ in Lemma \ref{6} and $\epsilon_i\rightarrow 0$ as $i\rightarrow \infty$. 
	\begin{lem} \label{95} Fix a symplectic form $\omega_Y$ on $Y$. Let $J'$ be an almost complex structure tamed uniformly by $\omega_Y$ and $\omega_i$ for $i\gg 1$ with 
		\begin{align*}
			\mathbf{C}_1^{-1} |X|^2_{g_Y}\geq \omega_Y(X,J'X) \geq \mathbf{C}_1|X|^2_{g_Y}  \\
			\mathbf{C}_2^{-1} |X|^2_{g_i}\geq {\omega_i}(X,J'X) \geq \mathbf{C}_2|X|^2_{g_i},
		\end{align*}
		where $g_i$ and $g_Y$ is the Riemannian metric associate to $\omega_i$ and $\omega_Y$. Let $u\in \mathcal{U}_i$. 
		There exist a compact subsets $K\subseteq X$ and  constants $C,C'>0$ both independent of $i$ such that if $L_u\subseteq  X\backslash K$ and $L_u$ bounds a stable $J'$-holomorphic disc then
		\begin{enumerate}
			\item if the relative class is $\beta_0$, then the symplectic area with respect to $\omega_Y$ is bounded above by $C\epsilon_i$.
			\item There exists $C'>0$ such that if the symplectic area (with respect to $\omega_Y$)  $\leq C'$ and the Maslov index is $0$ or $2$. Then its relative class is a multiple of $\beta_0$.
			In other words, all $J'$-holomorphic discs (of Maslov index $2$) with boundary on $L_u$ of class other than a multiple of $\beta_0$ have symplectic area $\geq C'$. 
		\end{enumerate}

	\end{lem}
	\begin{proof}
		Now assume that $L_u\in \mathcal{U}_i$, $i\gg 1$. Let $f:(D^2,\partial D^2)\rightarrow (X,L_u)$ be a $J'$-holomorphic disc of class $\beta_0$. 
		\begin{align} \label{estimate}
			\int_{\beta_0}\omega_i=\int_{D^2}f^*\omega_i\geq \mathbf{C}_2\int_{D^2}|df|^2_{g_i} d\mbox{vol}&\geq \mathbf{C}_2C_2 \epsilon_i^{-20} \int_{D^2}|df|^2_{g_Y} d\mbox{vol}\notag \\ &=\mathbf{C}_2C_2 \mathbf{C}_1^{-1}\epsilon_i^{-20} \int_{D^2}f^*\omega_Y.
		\end{align} The first equality is because that $L_u$ is a Lagrangian with respect to $\omega_i$. The first inequality is because that $J'$ is tamed and the second inequality follows from Lemma \ref{bound}. Since $L_u$ is Lagrangian, the first term in (\ref{estimate}) can be computed via any representative, say $\{|w|\leq \epsilon<\epsilon_i\}$ such that (\ref{bound'}) holds. Straight-forward calculation shows that 
		\begin{align} \label{estimate'}
			|\int_{\beta_0}\omega_i|\leq 
			\mathfrak{C}\epsilon_i^{-20}\epsilon_i,
		\end{align} for some $\mathfrak{C}$. Then combining (\ref{estimate})(\ref{estimate'}) together, we have the symplectic area of $f$ with respect to $\omega_Y$ is bounded by $C\epsilon_i$, where $C=\mathfrak{C}\mathbf{C}_1\mathbf{C}_2^{-1}C_2^{-1}>0$ independent of $i$. This proves the first part of the lemma. Choose the compact set $K'\subset X$ such that $Y\backslash K'$ is diffeomorphic to $Y_{\mathcal{C}}$. Then $H_2(Y\backslash K',L_u)\cong \mathbb{Z}\langle\beta_0, D\rangle$ and thus any holomorphic disc of class other than $n\beta_0+mD$ with boundary on $L_u$ would intersect $K'$. From the monotonicity of holomorphic discs (say \cite[Proposition 4.4.1]{S6}), the symplectic area with respect to $\omega_Y$ of a holomorphic disc of class other than $n\beta_0+mD$ will be bounded below by some constant $C''$.
		Now assume that the holomorphic disc is of relative class $n\beta_0+mD$ and is contained in $Y\backslash K'$. First we assume that it is of Maslov index zero, then 
		$n=-2md$. Then the symplectic area with respect to $\omega_i$ is given by $-2md\int_{\beta_0}\omega_i+m\int_D\omega_i= m(\int_D\omega_i-2d\int_{\beta_0}\omega_i)>0$. From (\ref{estimate'}), $m>0$ if $i$ is chosen large enough such that $\epsilon_i< \int_D\omega_Y/2d\mathfrak{C}$ from (\ref{estimate'}). Moreover, we have that $\int_{D^2}f^*\omega_Y> \frac{1}{2}\int_{D}\omega_Y$ if $\epsilon_i<\int_D\omega_Y/4d\mathfrak{C}$. 
		The case of Maslov index zero is similar.   
		Then one can take $C'=\mbox{max}\{\frac{1}{2}\int_D\omega_Y, C''\}$ and  $K=\bar{U}_i$, $i$ large enough such that $\epsilon_i<\min\{C'/C,\int_{D}/4d\mathfrak{C}\}$. This finishes the proof of the lemma. 
	\end{proof}
	
	The following is part of \cite[Conjecture 7.3]{A}. It is expected to be achieved via the technique of neck-stretching while here we give a rather elementary proof. 
	\begin{thm} \label{115}
		There exists $I>0$ such that $n_{\beta_0}(u)=1$ for $L_u\in \mathcal{U}_i^c$, $i>I$. 
	\end{thm} 
	\begin{proof} 
		From \cite[Theorem 3.3]{HSVZ}, the complex structure $J$ and $J_{\mathcal{C}}$ are asymptotic to each other and thus coincide on $D$. Therefore, there exists a $1$-parameter family of almost complex structures $J_t$, $t\in [0,1]$ on $Y$ such that 
		\begin{enumerate}
			\item $J_t$ tamed with $\omega_Y$, for all $t$.
			\item $J_t=J$ the standard complex structure on the compact set $K$ (in Lemma \ref{95}).
			\item  $J_0=J$ and $J_1=J_{\mathcal{C}}$ on $\mathcal{U}_i^c$ for $i\gg 1$. 
		\end{enumerate}
		From Lemma \ref{95}, any $J_1$-holomorphic discs in class $\beta_0$ has small symplectic area if the boundary condition is close enough to $D$. From the monotonicity, one can assume that all such discs falls in $\mathcal{U}_i^c$ and can compute $n_{\beta_0}(u)$ with respect to $J_1$ simply via the Calabi model. 
		By maximal principle, the model special Lagrangian (see Section \ref{91}) bound exactly one simple holomorphic disc, which is the disc along the fibre $X_{\mathcal{C}}\rightarrow D$ with boundary on the unit norm circle along the fibre. This holomorphic disc is regular by \cite[Theorem 2]{HLS}. Thus, $n_{\beta_0}(u)=1$ if computed with respect to $J_{\mathcal{C}}$.       	
	
		Next we want to understand how $n_{\beta_0}(u)$ varies with respect to the $J_t$. Assume that it is not constant with respect to $t$, there exists $J_{t_0}$-holomorphic discs $f_i:(\Sigma_i,\partial \Sigma_i)\rightarrow (X,L_u)$, $i=1,2$ for the same $t_0$ of relative class $\beta',\alpha\in H_2(Y,L_u)$ of Maslov index two and zero such that $\beta_0=\beta_1+\alpha$. This leads to a contradiction since 
		\begin{align*}
			C\epsilon_i> \int_{\Sigma_1}f_1^*\omega_Y+\int_{\Sigma_2}f_2^*\omega_Y> C_2+ 0\geq C\epsilon_i.
		\end{align*}
		Here the first inequality follows from the first part of  
		the Lemma \ref{95} the second inequality follows from the second part of the Lemma \ref{95}.
		
	\end{proof}

	Now we are ready to prove the tropical/holomorphic correspondence of superpotentials. Similar results of counting holomorphic cylinders in rigid analytic geometry was established by Yue Yu \cite{Y4}. 
	\begin{thm} \label{163}
		Let $u\in B_0$ generic and $\beta\in H_2(X,L_u)$ with $\beta\cdot D=1$, then   
		\begin{align*}
			n_{\beta}(u)=n^{trop}_{\beta}(u).
		\end{align*}
	\end{thm}
	\begin{proof}
		For $\beta\in H_2(X,L_u)$, recall that $l(\beta;u)<\infty$  from Lemma \ref{well-defined}. We will prove the theorem by induction on $l(\beta;u)$. For $l(\beta;u)=1$, then  $\beta=\beta_0$ and only one tropical disc contributes to $n^{top}_{\beta_0}(u)=1$. Thus, the theorem follows from Theorem \ref{115}. Assume that the theorem holds for $l(\beta;u)\leq k$. Now let $\beta\in H_2(X,L_u)$ and $n_{\beta}(u)\neq 0$. Let $(h,w,T)$ be a tropical disc of Maslov index two ending at $u$ such that $[h]=\beta$. Let $v$ is the vertex of the unique unbounded edge of $T$ and $T'$ is the complement of the unbounded edge. Then $(h|_{T'},w|_{T'},T')$ is a union of admissible tropical discs ending at $\phi(v)$, say $(h_0,w_0,T_1),(h_1,w_1,T_1)\cdots (h_l,w_l,T_l)$. Exactly one of them has Maslov index two say $h_0$ and rest of them are of Maslov index zero. Since every tropical disc of Maslov index two ending at $h(v)$ can enlongate the edge adjacent to $h(v)$ to $u$ and get a tropical disc of Maslov index two ending at $u$. Thus $\alpha([h_0],h(v))\leq k$ and the induction hypothesis implies that $n_{[h_0]}(\phi(v))=n^{trop}_{[h_0]}(\phi(v))$. Then the theorem follows from Theorem \ref{116} and the fact that the wall-crossing formula for tropical superpotential (see \cite{G9}) and wall-crossing formula for Floer theoretic superpotential (see Lemma \ref{WCF}) coincide.
	\end{proof} 
	
	\begin{rmk}
		For a toric Fano surface, there are mirrors constructed by Carl-Pumperla-Siebert \cite{CPS} and family Floer mirror \cite{F0}\cite{T4}\cite{A2} of the special Lagrangian fibration. Theorem \ref{116}, Theorem \ref{163} lay out the foundation for the equivalence of the mirrors.
	\end{rmk}
	The mirror of $(Y,D)$ constructed in \cite{CPS} is the fibrewise compactification of the Hori-Vafa superpotential. Recall that the Hori-Vafa superpotential $W_{HV}$ coincides with the superpotential of the Lagrangian torus fibres \cite{CO}. This motivates the following conjecture:
	\begin{conj} When $Y$ is a toric Fano surface,
		there exists a special Lagrangian torus fibre $L_u$ Hamiltonian isotopic to a moment torus fibre. In particular, $W(u)=W^{HV}$.
	\end{conj}

	\section{On the Relative Gromov-Witten Invariants} \label{1}
	The theory of symplectic relative Gromov-Witten invariants are developed by Ionel-Parker \cite{IP2}, Tehrani-Zinger \cite{TZ} and the subsequent works. The definitions are technical and a priori it is hard to establish the equivalence with the relative Gromov-Witten invariants defined in algebraic geometry. Under the setting of del Pezzo surfaces with special Lagrangian fibration, we here propose another approach below and the tropical analogue. 
	
	Recall that the monodromy of the special Lagrangian fibration around $\infty$ is conjugate to $\begin{pmatrix} 1 & d \\ 0 & 1\end{pmatrix}$. This motivates the following definition:
	\begin{definition}
		An affine ray $l$ labeled by $\partial \gamma$ is called admissible if
		the tangent of $l$ is monodromy invariant around the infinity. We say a tropical disc of Maslov index zero is admissible if the corresponding ray in the scattering diagram is admissible. 
		A tropical curve is admissible if it has a unique unbounded edge and the unbounded edge is admissible.
	\end{definition}
	Let $C_{\partial \gamma}$ be the parallel transport of the Lefschetz thimble of infinity to $u$. Set $\bar{\gamma}=\gamma+C_{\partial \gamma}\in H_2(Y,\mathbb{Z})/\mbox{Im}(H_2(L_u,\mathbb{Z})\rightarrow H_2(Y,\mathbb{Z}))$, where we abuse the notation and denote $\gamma$ its image under $H_2(X,L_u)\rightarrow H_2(Y,L_u)$. Since $L_u$ is contractible in $Y$, $\bar{\gamma}$ provides a well-defined curve class of $Y$.
	
	\begin{lem}\label{bounded energy}
		Let $l_{\mathfrak{d}}$ be a ray in the scattering diagram $\mathfrak{D}$ such that $f_{\mathfrak{d}}\neq 1$ and admissible. Fixed a K\"ahler form $\omega_{Y}$ and $f:(D^2,\partial D^2)\rightarrow (X,L_u)$ representing $\gamma_{\mathfrak{d}}$ for $u\in l_{\mathfrak{d}}$. Then $\int_{D^2}|df|^2_{g_Y}d\mbox{vol}<C$ for some constant $C>0$ independent of $u\in l_{\mathfrak{d}}$. 
	\end{lem}
	\begin{proof}
		Choose $\omega_i$, $i\gg 1$ such that $L_u$ is Lagrangian with respect to $\omega_i$. In particular, $\int_{\beta_0}\omega_i$ and $\int_{\gamma_{\mathfrak{d}}}\omega_i$ is independent of the representative of $\beta_0$ and $\gamma_{\mathfrak{d}}$. Then one has
		\begin{align*}
			\int_{\bar{\gamma}_{\mathfrak{d}}}\omega_i=k\int_{\beta_0}\omega_i+\int_{\gamma_{\mathfrak{d}}}\omega_i>\int_{\gamma_{\mathfrak{d}}}\omega_i.
		\end{align*}Here $k\in \mathbb{N}$ is the divisibility of $\partial \gamma$ and $u$ close enough to $D$. The last inequality is because there exists a holomorphic discs representing $\beta_0$ by Theorem \ref{115}. Then from Lemma \ref{bound}, one has 
		\begin{align*}
			k\epsilon_i^{-20}=\int_{\bar{\gamma}_{\mathfrak{d}}}\omega_i>\int_{\gamma_{\mathfrak{d}}}\omega_i\geq C_2\epsilon_i^{-20}\int_{D^2}f^*\omega_Y.
		\end{align*} Here the first equality comes from $[\omega_i]=k_ic_1(Y)$, $D\cdot \bar{\gamma_{\mathfrak{d}}}=k$ and the lemma follows. 
		%
		%
	\end{proof}

	Let $l$ be an admissible affine ray. Assume that $u_t\in l$, $t\in [0,\infty)$ be a parametrization of $l$ and assume that $\tilde{\Omega}(\gamma;u_t)\neq 0$ along $l$, for some $\gamma$ and $\partial \gamma$ is monodromy invariant. We will first have a compactification of the $1$-parameter family of the moduli spaces of holomorphic discs in relative class $\gamma$,
	\begin{align*}
		\bigcup_{t\in [0,\infty)}\mathcal{M}_{\gamma}(X,L_{u_t}).
	\end{align*} 
	Let $p:\mathcal{Y}=Y \times \mathbb{P}^1 \rightarrow \mathbb{P}^1$ and identify $L_{u_t}'=L_{u_t}\times \{t^{-1}\}\subseteq Y\times \{t^{-1}\}$. Consider $\mathcal{L}$ be the closure of $\cup_{t\in [0,\infty)}L_{u_t}'\subseteq \mathcal{Y}$, which is a $3$-manifold by adding an $S^1$ in $D\times \{0\}$, denoted by $S^1_l$, to $\cup_{t\in [0,\infty)}L_{u_t}'$. We will identify $S^1_l$ as a subset of $D$. 
	\begin{lem}
		The $3$-manifold $\mathcal{L}$ is totally real in $\mathcal{Y}$. 
	\end{lem}
	\begin{proof}
		Let $J$ be the product (almost) complex structure of $\mathcal{Y}$, it suffices to prove that $JT\mathcal{L}\cap T\mathcal{L}=\{0\}$ for every $p\in \mathcal{L}$. If $p\in \cup_{t\in [0,\infty)}L_{u_t}'$, 
		we may assume
		there are $v_1$, $v_2\in T_pL_{u_t}'$ and $v_3$, $v_4\in T_p\mathcal{L}$
		which project to non-zero vectors in $T\mathbb{P}^1\subseteq T(B\times \mathbb{P}^1)$ such that
		\begin{align*}
			J(v_1+v_3)=v_2+v_4 \quad \mbox{or} \quad Jv_3-v_4=v_2-Jv_1.
		\end{align*} Notice that on one hand $v_2-Jv_1\in TX_t$ and projects to zero vector in $T\mathbb{P}^1$. On the other hand, the second assumption of $v_i$ guarantees that $p_*v_3 \parallel p_*v_4 \in T\mathbb{P}^1$. Therefore, we have $p_*Jv_3\perp p_*v_4 \in T\mathbb{P}^1$. In particular, $p_*(Jv_3-v_4)\neq 0\in T\mathbb{P}^1$ and leads to a contradiction. 
		
		Now let $x\in \mathcal{L}\backslash \cup_{t\in [0,\infty)}L_{u_t}'$. Set $v_1,v_2,v_3\in T_x\mathcal{L}$ basis. One may assume that $v_3\in T_xD$. $p_*v_1\neq 0$, $p_*v_2=p_*v_3=0$, where $\pi:X\rightarrow B$. If $v\in JT_x\mathcal{L}\cap T_x\mathcal{L}\backslash\{0\}$, then $p_*Jv_1=Jp_*v_1\neq 0$ and not linear combination of $p_*v_i$ which leading to a contradiction. 
		
	\end{proof}
	In particular, one can study the holomorphic discs with boundary on $\mathcal{L}$ and the associate Cauchy-Riemann operator is Fredholm. Given $u_{t_i}\in l_{\mathfrak{d}}$, $t_i\rightarrow \infty$. There exist holomorphic maps representing $k\gamma$, say
	\begin{align*}
		f_i:(\Sigma_i,\partial \Sigma_i)\rightarrow (X,L_{u_{t_i}})\cong (X\times \{t^{-1}_i\},L'_{u_{t_i}})\subseteq (\mathcal{Y},\mathcal{L}),
	\end{align*} where $\Sigma_i$ are bordered Riemann surface of genus zero by the correspondence theorem (Theorem \ref{116}). Moreover, we have the uniform energy bound for $\int_{\Sigma_i}f_i^*\omega_Y$ by Lemma \ref{bounded energy}. Therefore, the compactness theorem still applies \cite{F4} and $f_i$ converge to a holomorphic map $f_{\infty}:(\Sigma_{\infty},\partial \Sigma_{\infty})\rightarrow (\mathcal{Y},\mathcal{L})$ representing $\bar{\gamma}$. Conversely, given a stable holomorphic disc $f:(\Sigma,\partial\Sigma)\rightarrow (\mathcal{Y},\mathcal{L})$. The composition with the projection to the $\mathbb{C}$-factor gives a holomorphic map $f':(\Sigma,\partial\Sigma)\rightarrow (\mathbb{C},\mathbb{R}_{\geq 0})$. By maximal principle, the image $f'(\partial\Sigma)$ is a point. In other words, every such holomorphic disc is contains in a fibre of projection $p$ and
	\begin{align}\label{21}
		\mathcal{M}_{\tilde{\gamma}}(\mathcal{Y},\mathcal{L})      \backslash \bigcup_{t\in [0,\infty)}\mathcal{M}_{\gamma}(X\times \{t^{-1}\},L_{u_t}')\neq \emptyset.
	\end{align}   
	Next we want to know the shape of the curves in the (\ref{21}). 
	\begin{lem} \label{1000}
		The curves in (\ref{21}) compose with the projection $p$ are rational curves with maximal tangency $w$ with $D$, where $w=\bar{\gamma}\cdot D$. 
	\end{lem} 
	\begin{proof} Let $f:(\Sigma,\partial \Sigma)\rightarrow (\mathcal{Y},\mathcal{L})$ is an element in (\ref{21}). Since it falls in a fibre of $p$, we will not distinguish $f$ and $p\circ f:(\Sigma,\partial \Sigma)\rightarrow (Y,S^1_{\mathfrak{d}})$. We will prove that $f(\partial \Sigma)$ is a point and $f(\Sigma)\cap D$ is a point. 
		Assume $f_i:(\Sigma_i,\partial \Sigma_i)\in \bigcup_{t\in [0,\infty)}\mathcal{M}_{\gamma}(L_{u_t})$ is a sequence converging to $f$.  From the continuity, the boundary $f(\partial\Sigma)\subseteq \mathcal{L}\backslash \bigcup_{t\in [0,\infty)}L_{u_t}'\subseteq D$, which is complex. This implies that $f(\Sigma)$ is tangent to $D$ along its boundary. The unique continuation theorem implies that a component of $f(\Sigma)$ coincides with $D$ if $f(\partial \Sigma)$ is not a point. Since $\Sigma$ is a union of rational curve and a disc, $f(\partial \Sigma)$ is a point. By maximum principle, the disc component is a constant and $f(\Sigma)$ is an image of tree of rational curve(s).
		Notice that $f_i$ are of Maslov index zero, so $f_i(\Sigma_i)$ avoids $D$. 
		Due the convexity in a neighborhood of $D$ and maximum principle, $f(\Sigma)$ can only intersect $D$ along $S^1_l$ and no bubbling occurs in a collar neighborhood of $f(\Sigma)$. In particular, $f(\Sigma)$ is irreducible. $f_i(\partial \Sigma_i)$ converging to a point implies that $f(\Sigma)$ intersect $D$ at a unique point. In other words, the $f(\Sigma)$ is a rational curve with maximal tangency.    
	\end{proof}
	The following lemma is the first step to make sense of the weighted count of admissible tropical curves.
	\begin{lem} \label{finite rays}
		Given a fixed $w\in \mathbb{N}$,
		there are only finitely many disjoint admissible affine rays $l$ such that 
		\begin{enumerate}
			\item $\tilde{\Omega}(\gamma,u)\neq 0$ for some $\gamma\in H_2(X,L_u)$ along $l$. 
			\item  $\partial \gamma=-w\partial \beta_0$. 
		\end{enumerate}
	\end{lem}
	\begin{proof}
		From Lemma \ref{1000}, every such admissible ray  $l$ gives rise to a rational curve intersecting $D$ at a single point in $S^1_l$ with tangency $w$. Then such point is a $w$-torsion point of $D$ and there is only $w^2$ of such points. It is a classical fact in algebraic geometry that  there are only finitely many such rational curves with tangency condition $w$. Indeed, there only finitely many effective curve classes of $Y$ which intersection number with $D$ is $w$ since $D$ is ample. If there are infinitely many rational curves with tangency condition $w$ (notice that $Y$ has an integral almost complex structure), there exists a positive family of such rational curves $g:\mathbb{P}^1\times C\twoheadrightarrow Y$ with $g^{-1}(D)\subseteq \{\infty\}\times C$. Then $g^*\Omega$ is a non-zero holomorphic $2$-form on $(\mathbb{P}^1\backslash \{\infty\})\times C$ with a simple pole along $\{\infty\}\times C$. Contracting the pull-back of a non-zero vector field on an open subset of $C$ gives a non-zero meromorphic $1$-form on $\mathbb{P}^1$ with only a simple pole at infinity, which leads to a contradiction. 
		Since $S^1_{l}\cap S^1_{l'}=\emptyset$ if $l, l'$ disjoint in a neighborhood of infinity, there are only finitely many such admissible affine rays.
	\end{proof}
	
	The following theorem is to make sense of weighted count of admissible tropical curves. 
	\begin{thm} \label{finite trop rel}
		There are only finitely many admissible tropical curves representing $\beta\in H_2(X,\mathbb{Z})$ with $\beta\cdot D\leq w$, for a fixed $w\in \mathbb{N}$.
	\end{thm}
	\begin{proof}
		Given an admissible affine ray $l$ in Lemma \ref{finite rays}, it suffices to prove that there are finitely many admissible tropical curves
		\begin{enumerate}
			\item with the image of the unique unbounded edges intersect $l$. 
			\item  representing a curve class with intersection number $w$ with $D$. 
		\end{enumerate}
		We will use the notation in the proof of Lemma \ref{MI2 asym}. Let $u=e^{r+i\theta}$ with $\theta\in [0,2\pi)$ and $r<-kR_0$. Assume that there are two affine rays $l,l'$ such that
		\begin{enumerate}
			\item intersect at $u$ and labeled by $m+n\frac{d}{2\pi i}\log{u}$ and $m'+n'\frac{d}{2\pi i}\log{u}$ for $m,n,m',n'\in \mathbb{Z}$,
			\item the corresponding parabolas have their vertices $(r_0, \frac{2m\pi}{nd})$ and $(r_0', \frac{2m'\pi}{n'd})$ with $r_0, r_0'> -R_0$ and $n=-n'$. 
		\end{enumerate} We claim that $m+m'> kR_0$. 
		Without loss of generality, we may assume that $\frac{2m'\pi}{n'd}<\theta< \frac{2m\pi}{nd}$. This implies that $l$ spirals into the infinity (a.k.a $u=0$) clockwisely and $l'$ spirals into the infinity counter-clockwisely. This implies that $m,n$ have the same sign and $m',n'$ have opposite signs. Assume that $l$ is the tropical disc representing $\gamma$ then
		\begin{align*}
			\int_{\gamma}-\frac{du}{u}\wedge dx= -m(r-r_0)-\frac{nd}{4\pi}(2r\theta-2r_0\theta_0)
		\end{align*} is increasing along $l$ or $-rm(1+\frac{nd}{2m \pi}\theta)>0$. Thus $m>0$ and $n>0$ since $mn>0$. Since $m',n'=-n<0$ have different signs, we have $m'>0$ as well.  
		From equation (\ref{parabola}), one has 
		\begin{align*}
			|\theta-\frac{2m\pi}{nd}|= (r^2-r_0^2)^{\frac{1}{2}}> (k^2-1)^{\frac{1}{2}}R_0
		\end{align*} and similar one for $m',n'$. Therefore, 
		\begin{align*}
			\frac{2\pi}{nd}(m+m')=\frac{2\pi m\pi}{nd}-\frac{2m'\pi}{n'd}=(\frac{2\pi m\pi}{nd}-\theta)+(\theta-\frac{2m'\pi}{n'd})> 2(k^2-1)^{\frac{1}{2}}R_0
		\end{align*} and the claim is proved since $n,d\in \mathbb{N}$.
		
		Now assume $(h,T,w)$ is an admissible tropical curve such that the image of unbounded edge intersect $l$ and
		the unique vertex $v$ adjacent to the unbounded edge falls in $U_{\infty}$. If we write $h(v)=e^{r+i\theta}$, then $r<-R_0$. Let $e,e'$ be the other two edges of $T$ adjacent to $v$ and $l,l'$ be the affine rays extending the image $h(e),h(e')$. Then $l,l'$ satisfies the second assumption in the claim by claim in Lemma \ref{MI2 asym}. Thus, the claim implies that $\overline{[h]}\cdot D=m+m'>kR_0$ if $r<-kR_0$. In other words, $h(v)$ falls outside a neighborhood of infinity depending on $w$.
		
		Now assume that there are infinitely many distinct admissible tropical curves $(h_i,T_i,w_i)$ representing $\beta$ with $\beta\cdot D=w$ and the unique unbounded edges coincide in a neighborhood of the infinity. Choose $u_{\infty}\in l$ close to $D$ enough such that $n_{\beta_0}(u_{\infty})=1$ by Theorem \ref{115}. In particular, there exists a holomorphic disc representing $\beta_0$ ending on $L_{u_{\infty}}$. Choose $\omega_i$ such that $L_{u_{\infty}}$ is a Lagrangian and then $\int_{\beta_0}\omega_i>0$. Chopping off the unbounded edges at $u_{\infty}$ from $(h_i,T_i,w_i)$ and give rise to infinitely many distinct tropical discs $(h_j',T_j',w_j')$ ending at $u_{\infty}$ such that 
		\begin{align*}
			\int_{[h_j']}\omega_{TY}=\int_{[h_j']}\omega_i=\int_{\beta}\omega_i-\int_{\beta_0}\omega_i<\int_{\beta}\omega_i,
		\end{align*} bounded by a fixed number. However, there are only finitely many tropical discs (up to enlongation of the edge adjacent to the root) and leads to a contradiction. This finishes the proof of the theorem. 
	\end{proof}
	The following is a combination of Theorem \ref{finite trop rel} with the correspondence theorem (see Theorem \ref{116}):
	\begin{cor}
		Let $l$ be an admissible ray and $\gamma\in H_2(X,L_u)$, for $u\in l$ with $\partial\gamma$ monodromy invariant around the infinity. Then  
		\begin{align*}
			\tilde{\Omega}(\gamma;l):=\lim_{u\rightarrow \infty, u\in l}\tilde{\Omega}(\gamma;u)
		\end{align*}
		exists. 
	\end{cor}
	With the above theorem, we can have the following definition:
	\begin{definition}
		Given a curve class $\beta\in H_2(Y,\mathbb{Z})$, define the tropical relative Gromov-Witten invariant of genus zero
		\begin{align*}
			N^{Y/D,trop}_{0,\beta}=\sum_{(\gamma;l):\bar{\gamma}=\beta} \tilde{\Omega}(\gamma;l).
		\end{align*}	 
	\end{definition}
	\begin{rmk}
		One would expect that there are $w$ such admissible affine lines and each corresponds to $w$ of the $w$-torsion points.
	\end{rmk}
	We expect that all the rational curves with maximal tangency condition are limiting of holomorphic discs, i.e. elements in (\ref{21}).
	Therefore, the compactification of the family of moduli spaces
	\begin{align*}
		\bigcup_{t\in [0,\infty)}\mathcal{M}_{\gamma}(X,L_{\phi(t)})
	\end{align*}
	gives a cobordism between the moduli spaces of holomorphic discs $\mathcal{M}_{\gamma}(L_u)$ and the moduli space of holomorphic curves with maximal tangency. This leads to the following conjecture
	\begin{conj} \label{open to closed}
		The symplectic count of holomorphic curves with maximal tangency is the sum of limit of certain open Gromov-Witten invariants. 
	\end{conj}
	This is the long-term conjectural relation from open to relative Gromov-Witten invariants, which motivated the earlier work of Li-Song \cite{LS2}. 
	If one can further identify the tropical geometry on the base of the special Lagrangian fibration introduced here with the one developed in the Gross-Siebert program. 
	Via the tropical/holomorphic correspondence, it also implies the equivalence of the symplectic relative Gromov-Witten invariants with the algebraic relative Gromov-Witten invariants. We will see it is indeed the case for $Y=\mathbb{P}^2$ in the next section. 
	Although the result is more restrictive and not applies to all projective manifold, the usage of tropical/holomorphic correspondence avoids the difficulty of comparing the fundamental cycles directly. We will explore in the direction in the future.

	%
	%
	%
	%
	%
	
	\section{Tropical Geometry of $\mathbb{P}^2$}
	\subsection{Complex Affine Structure of $\mathbb{P}^2$} Strominger-Yau-Zaslow conjecture suggests that the mirror is given by the dual torus fibration of the special Lagrangian fibration. 
	Due to the difficulty in analysis of the existence of special Lagrangian fibration, Kontsevich-Soibelman \cite{KS1}, Gross-Siebert \cite{GS1} proposed an algebraic alternative approach to construct the mirror. In the algebraic construction, one starts with a toric degeneration and there is a natural affine manifold with singularities $B_{GS}$ induced from the dual intersection complex. The affine manifold $B_{GS}$ is the algebraic analogue of the base for special Lagrangian fibration. From the scattering diagram on $B_{GS}$, which is the algebraic analogue of information of holomorphic discs with boundary on special Lagrangian torus fibres, one can reconstruct the mirror. 
	
	The case of the pairs $(Y,D)$, where $Y$ is a toric del Pezzo surface and $D$ is a smooth anti-canonical divisor is considered by Carl-Pumperla-Siebert \cite{CPS}. 
	Below we will describe the affine manifold for $\mathbb{P}^2$ in \cite{CPS}. The underlying space is $\mathbb{R}^2$. There are three singularities with local monodromy conjugate to $\begin{pmatrix} 1 & 1\\ 0 & 1 \end{pmatrix}$ locating at $u_1=(\frac{1}{2},\frac{1}{2})$, $u_2=(0,-\frac{1}{2})$ and $u_3(-\frac{1}{2},0)$. To cooperate with the standard affine structure of $\mathbb{R}^2$ for computation convenience, they introduce cuts and the affine transformation as follows: let
	\begin{align*}
		l_1^+&=\{(\frac{1}{2},y)|y\geq \frac{1}{2} \}  \\
		l_1^-&=\{(x,\frac{1}{2})|x\geq \frac{1}{2} \}  \\
		l_2^+&=\{(x,-\frac{1}{2})|x\geq 0  \}          \\
		l_2^-&=\{(-t, -\frac{1}{2}-t)|t\leq -\frac{1}{2}| t\geq 0 \}          \\
		l_3^+&=\{ (-\frac{1}{2}-t,-t)| t\geq 0 \}  \\
		l_3^-&=\{(-\frac{1}{2}, y )|y \geq 0 \}    
	\end{align*} 
	Discard the sector bounded by $l_i^+,l_i^-$, then glue the cuts by the affine transformations $\begin{pmatrix} 2 & 1\\ -1 & 0 \end{pmatrix}$, $\begin{pmatrix} -1 & 4\\ -1 & 3 \end{pmatrix}$, $\begin{pmatrix} -1 & 1\\ -4 & 3 \end{pmatrix}$ and one gets the affine manifold in \cite{CPS}. See Figure \ref{fig:CPS}\footnote{The picture credit is due to Tsung-Ju Lee.} below. 
	\begin{figure}
		\centering
			\includegraphics[scale=0.8]{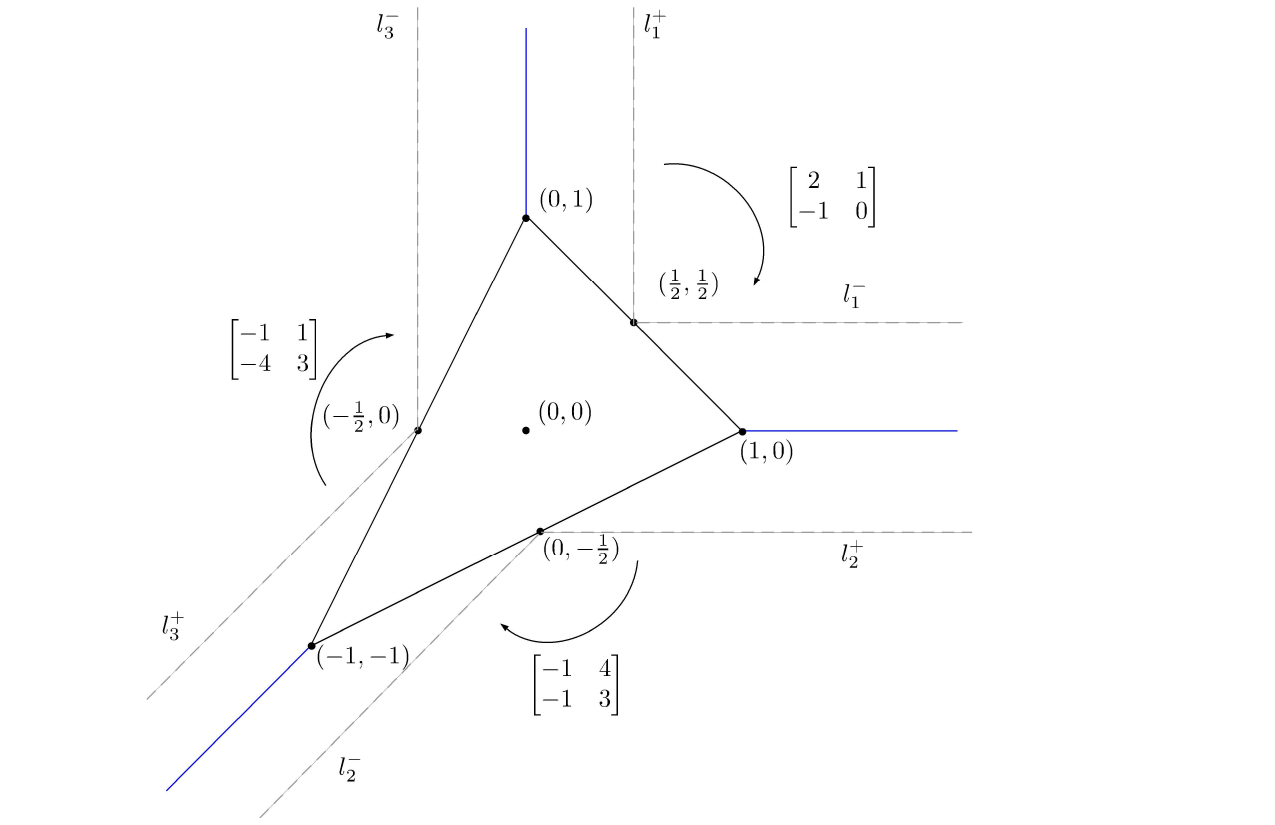}
		\caption{The integral affine structure in Carl-Pumperla-Siebert}
		\label{fig:CPS}
	\end{figure}
	
	\begin{definition}
		We will call the region bounded by $l_1^-$, $l_2^+$ and the affine segments connecting $(\frac{1}{2},\frac{1}{2}),(0,-\frac{1}{2})$ and the origin the $x$-region. Similar for $y$-region and $1/xy$-region. 
	\end{definition}
	\begin{definition}
		Let $u\in B$ and $v\in T_uB_{\mathbb{Z}}$. Assume that $u$ falls in $x$-region, then scale of $v$ is defined to be the $x$-component of $v$, i.e. $v\cdot (1,0)$. We define the scale for $u$ falls in $y$-region or $1/xy$-region similarly. 
	\end{definition}
	\begin{lem} \label{159}
		Let $(h,w,T)$ be a tropical disc of Maslov index zero with stop at $u$. Assume that $e$ is the edge adjacent to $u$ and $u'$ is the other end point of $h(e)$. Let $v\in T_{u'}B_{\mathbb{Z}}$ be the primitive tangent vector of $h(e)$. Then 
		\begin{enumerate}
			\item the scale of $v$ positive and 
			\item the parallel transport of $v$ along $h(e)$ has non-decreasing scale. 
		\end{enumerate}
	\end{lem}
	\begin{proof}
		We will prove the lemma by induction on the number of non-contracted edges. If $h$ has only one non-contracted edge, then it is the (multiple of the) initial disc and the first part of lemma follows from direct computation. Assume that $h(e)$ stay in the same region, then the scale of $v$ is invariant under the parallel transport. To prove the second part of the lemma, it suffices to consider the case when $h(e)$ go through the cut once. We will prove the case that it goes from $1/xy$-region to $x$-region and the other cases follow from the similar calculation. Let $v=(a,b)$, then one has $0<-a<-b$. Then after parallel transport to the  $x$-region, one has 
		\begin{align*}
			v\mapsto \begin{pmatrix} 3 & -4 \\ 1 & -1 \end{pmatrix} \begin{pmatrix}a \\b \end{pmatrix}= \begin{pmatrix} 3a-4b \\a-b \end{pmatrix} 
		\end{align*} and the later has scale $3(a-b)-b\geq 4>0$.    
		Now assume the lemma holds for all tropical discs of Maslov index zero with at most $k$ non-contracted edges. Let $(h,w,T)$ be a tropical disc of Maslov index zero with $k+1$ non-contracted edges. Then the first part of the lemma follows from the induction hypothesis. The second part of the lemma follows from the same calculation above.

	\end{proof}

	Generally the complex/symplectic affine structure associated to a special Lagrangian fibration is hard to compute due to various difficulties in analysis. One important feature of the special Lagrangian fibrations constructed in \cite{CJL} is that the equation for $X$ and $\check{X}$ can be both explicit. This is used to compute the complex affine structure.

	\begin{thm} \cite{LLL} \cite{B5} \label{169} Let $Y=\mathbb{P}^2$.
		The complex affine structure on $B$ induced from the special Lagrangian fibration on $X$ coincides with the the affine structure from Carl-Pumperla-Siebert \cite{CPS}.
	\end{thm} The above theorem together with the tropical/holomorphic correspondence makes the computation of the open Gromov-Witten invariants more explicitly in the next section.

	\subsection{Explicit Examples of Lower Degree for $\mathbb{P}^2$} \label{2}
	In this section, we will compute the some of the open Gromov-Witten invariants of Maslov index zero and Maslov index two.

	First we consider the open Gromov-Witten invariants for the Maslov index zero discs. 
	\begin{ex} \label{168}
		Given a rational curve $C$ of degree $d$ with maximal tangency to the elliptic curve $D$. Let $p=C\cap D$, then $p$ is a $3d$-torsion point of the elliptic curve $D$.  which are expected to recover the log Gromov-Witten invariants defined in Section \ref{1} in degree $d=1,2$ explicitly via the tropical/holomorphic correspondence.  
		For $d=1$, there are three tropical log curves and each with multiplicity three. Indeed, there nine $3$-torsion points and the tangent line at these $3$-torsion points are the only rational curves of degree one with maximal tangency. For $d=2$, there are three affine rays corresponding to $6$-torsion but not $3$-torsion. Three exists one tropical log curve in each direction with multiplicity $6$. There are three affine rays corresponding to $3$-torsion points. There are four tropical log curves with the only unbounded edges mapped to each of the affine ray (see Figure \ref{fig:1051}), with weights $3/4, -9/2, -9/2, 21/4$ respectively. From the multiple cover formula (\ref{multiple cover}), we have $(21/4+3/4)=-6$ (up to sign), which correspond to the $6$ rational curves of degree with maximal tangency and matches with the calculation by Takahashi \cite{T6}. 
	\end{ex}

	\begin{prop} \label{sup}
		Assume that $u_0$ is the origin in Figure \ref{fig:CPS}, then the superpotential $W(u_0)=x+y+\frac{1}{xy}$. 
	\end{prop}
	\begin{proof}
		One can scale the affine structure such that the three singular fibres are in a small neighborhood of infinity. Then the boundary divisor falls in a tubular neighborhood of degeneration of three hyperplanes. Up to a coordinate change, one can identify the three hyperplanes as the standard coordinate planes $\{x=0\},\{y=0\}, \{z=0\}$. Under the identification, $L_0$ is isotopic to the moment torus fibre. The obviously three tropical discs of Maslov index two have homology classes corresponds to the homology of Maslov index two discs of moment torus fibre via the isotopy. 
		
		To prove the theorem, it suffices to prove that there are no other tropical discs of Maslov index two. edge adjacent to $u_0$. Since every Maslov index zero has its scale equals or larger than one and scale is additive with respect to gluing of tropical trees, $(h,w,T)$ contains no sub-tropical discs of Maslov index zero. Following the same argument of Lemma \ref{159}, any affine ray leaving the original basic region the scale is positive and non-decreasing. Therefore, the above three tropical discs are the only tropical discs of Maslov index two. 
		
	\end{proof}
	
	\begin{prop}
		$L_{u_0}$ is a monotone Lagrangian torus in $\mathbb{P}^2$ for a suitable choice of $\omega_i$. 
	\end{prop}
	\begin{proof}
		Notice that there is a $\mathbb{Z}_3$-action on the extremal rational elliptic surface permuting the three singular fibres fixing the unique meromorphic volume form with simple pole along the $I_9$ fibre. From (\ref{HK rel}), this translates to a symplectomorphism of order three of $(X,\omega)$. Notice that we don't a priori assume that the equation of $D$ has an $\mathbb{Z}_3$-symmetry. Let $\sigma$ be a such symplectomorphism. 
		
		To see that 
		$L_0$ is a monotone Lagrangian torus,
		recall that one has the short exact sequence    
		\begin{align*}
			0\rightarrow H_2(\mathbb{P}^2,\mathbb{Z})\rightarrow H_2(\mathbb{P}^2,L_0;\mathbb{Z})\rightarrow H_1(L_0,\mathbb{Z})\rightarrow 0. 
		\end{align*}
		Denote $\beta_0,\beta_1,\beta_2\in H_2(\mathbb{P}^2,L_0;\mathbb{Z})$ representing the relative class representing the three tropical discs of Maslov index two in Proposition \ref{sup}. Then $\mu(\beta_i)=2$ and $\beta_i$ generates $H_2(\mathbb{P}^2,L_0;\mathbb{Z})$. Notice that $L_{u_0}$ is the fixed locus of $\sigma$ and the action of $\sigma$ permutes $\beta_0,\beta_1,\beta_2$, we have  $\omega(\beta_0)=\omega(\beta_1)=\omega(\beta_2)>0$. Since $\beta_0+\beta_1+\beta_2=H$, where $H$ is the generator of $H_2(\mathbb{P}^2,\mathbb{Z})$. Thus, $L_0$ is a monotone Lagrangian torus with respect to $\omega$.
	\end{proof}
	Together with Proposition \ref{sup},
	it is natural to expect the following statement:
	\begin{conj}
		$L_{u_0}$ is Hamiltonian isotopic to the monotone moment map torus of $\mathbb{P}^2$. 
	\end{conj}
	Here we provide a heuristic argument why the conjecture is expected to be true: 
	There are three tropical curves of multiplicity $3$ in Figure \ref{fig:CPS}. As explained in Example \ref{168}. these correspond to the tangent of the nine $3$-torsion points of $D$. As $D$ is close enough to a nodal curve, the special Lagrangian fibration on $X$ is collapsing \cite{CJL2}. For each above tropical curve, one can reconstruct the holomorphic curves by gluing three cylinders with a pair-of-pants as in \cite{P6} to get three rational curve avoiding $L_{u_0}$. Let $H$ is one of them and $H$ is of degree one since $H. D=3$. Then $L_{u_0}\subseteq \mathbb{P}^2\backslash (H\cup \sigma H\cup \sigma^2 H)\cong (\mathbb{C}^*)^2$. With a suitable Moser's trick, then $L_{u_0}$ is Hamiltonian isotopy to a monotone torus in $(\mathbb{C}^*)^2\cong T^*T^2$ with the standard symplectic form. The nearby Lagrangian conjecture for $2$-tori \cite{DGI} states that every exact Lagrangian in $T^*T^2$ is Hamiltonian isotopic to the zero section which corresponds to  the unique monotone moment fibre of $\mathbb{P}^2$.

	\begin{figure} \label{fig:1051}
		\begin{center}
			\includegraphics[height=1.5in,width=3in]{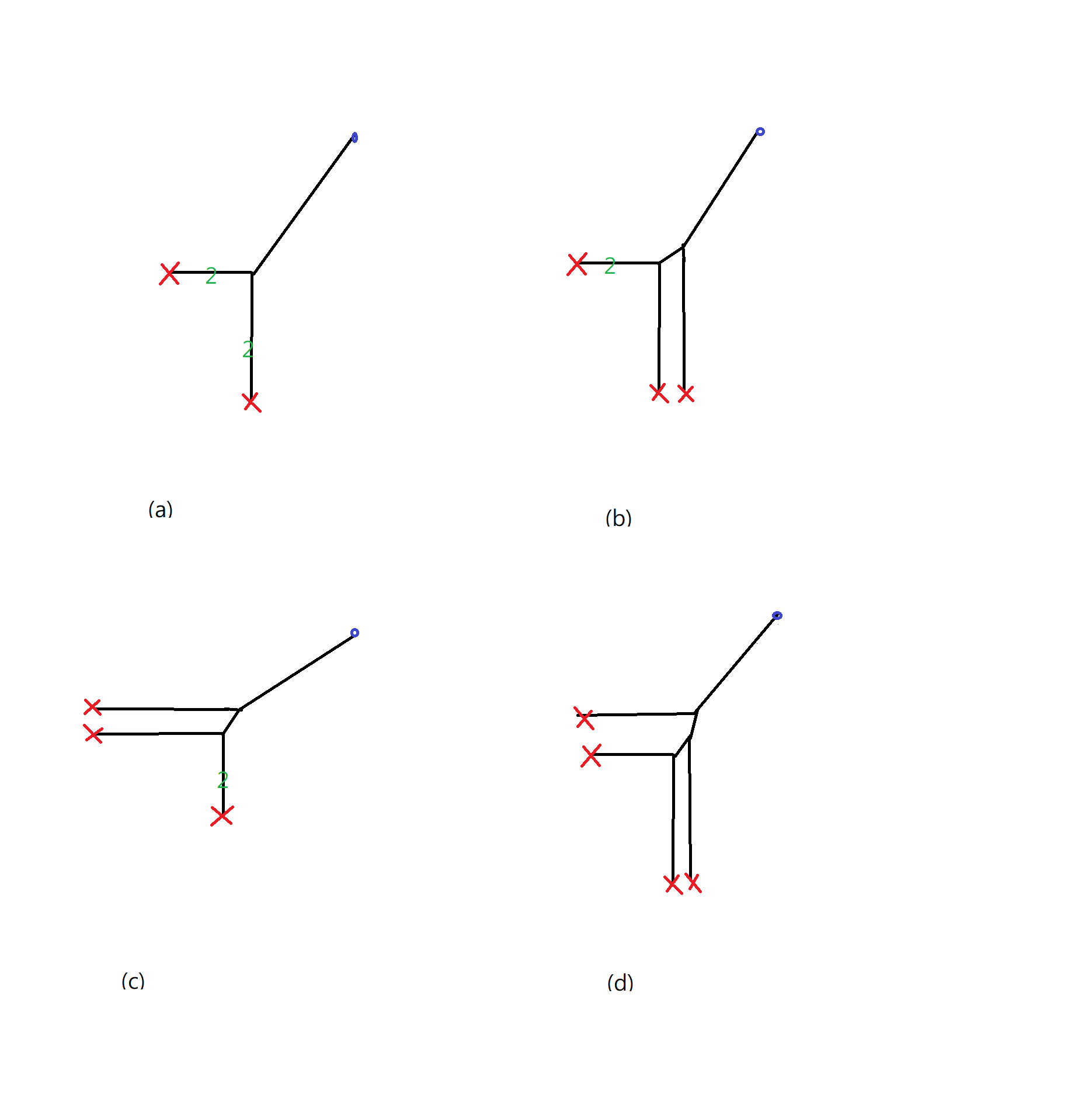}
			\caption{All the admissible tropical discs contribute to the calculation of Example \ref{168}. The bounded edges are contracted to a point and all the tropical discs have the same image.}
		\end{center}
	\end{figure}   
	
	\subsection{Connection to other works} 
	\subsubsection{From open to Relative Gromov-Witten Invariants of $\mathbb{P}^2$}
	Consider the wall-structure constructed by Carl-Siebert-Pumperla in the case of $(Y=\mathbb{P}^2, D)$, where $D$ is a smooth cubic curve. Let $f_{out}$ be the product of all slab functions associated to walls with tangent monodromy invariant near infinity. Then $f_{out}\in \mathbb{C}[[t]][x]$, where $x=z^{(-m_{out},1)}, t=z^{0,1}\in \mathbb{C}[\Lambda\oplus \mathbb{Z}]$. Here $m_{out}$ is the primitive vector pointing toward infinity and monodromy invariant around infinity. Tim Graefnitz proved the enumerative interpretation of the coefficients of $f_{out}$:
	\begin{thm}\cite{G8} \label{Gabele}
		With the above notation, then 
		\begin{align*}
			\log{f_{out}}=\sum_{d\geq 0}^{\infty} 3d \cdot  N^{\mathbb{P}^2/D}_{0,d} \cdot t^{3d}x^{3d}.
		\end{align*} Here $N^{\mathbb{P}^2/D}_{0,d}$ is the relative Gromov-Witten invariant of $1$-marked stable maps to $\mathbb{P}^2$ of genus zero, degree $d$ and maximal tangency with $D$ at a single unspecified point. 
	\end{thm}
	This can also interpreted as a tropical correspondence theorem of $(\mathbb{P}^2,D)$. Recall that Theorem \ref{169} identify the complex affine structure of the special Lagrangian fibration on $\mathbb{P}^2\backslash D$ and the affine structure in Carl-Pumperla-Siebert. In particular, this identifies the scattering diagram discussed in Section \ref{tropical geometry} and the one in \cite{CPS} and also the corresponding slab functions in the case of $\mathbb{P}^2$. Then Theorem \ref{116}, Theorem \ref{169} and Theorem \ref{Gabele} together implies the following folklore conjecture between open and closed invariants which echos the expectation in the work of Gross-Siebert \cite{GS3}\cite{GS4}:
	\begin{thm} 
		The algebraic relative Gromov-Witten invariant $N^{\mathbb{P}^2/D}_{0,d}$ can be computed via open Gromov-Witten invariants 
		\begin{align*}
			N^{\mathbb{P}^2/D}_{0,d}=\sum_{(\gamma;l):\partial\gamma=-d\partial \beta_0} \tilde{\Omega}(\gamma;l).
		\end{align*}
	\end{thm} 
	The first of this kind of results was due to Chan \cite{C2} by direct comparison of the Kuranishi structures of the relevant moduli spaces.  Here, we outline a complete new methodology. Indeed the idea here is also carried out such correspondence for all Looijenga pairs \cite{BCHL}. 
	It worth mentioned that recently Garrel-Ruddat-Siebert \cite{GRS} proved that the period integral of the (SYZ-fiber-normalized) volume form over the SYZ section in the Carl-Pumperla-Siebert mirror of $\mathbb{P}^2$ yields precisely the expression given in Theorem \ref{Gabele}.

	When $\partial \gamma=-d\partial \beta_0$, one can compute the quadratic refinement $c(\gamma;u)=(-1)^d$ via induction and relation 
	\begin{align*}
		c(\gamma_1,\gamma_2;u)=(-1)^{\langle \gamma_1,\gamma_2\rangle}c(\gamma_1;u)\c(\gamma_2;u).
	\end{align*} Together with \cite[Proposition 2.3.2, Lemma 1.2.2 ]{B5}, one has the some partial result on the open version of the Gopakumar-Vafa conjecture
	\begin{prop}
		Conjecture \ref{12} holds for $u$ near infinity and $\partial \gamma=-d\partial \beta_0$. 
	\end{prop}
	
	\subsubsection{Special Lagrangian as stable objects}
	It is generally a hard question to construct the stability conditions for Fukaya type categories and analyze the stable objects. A well-known (and vague) folklore conjecture motivated by the work of Thomas-Yau \cite{TY2} is the following:
	\begin{conj} \label{stab}
		The complex structures are stability conditions for Fukaya categories and special Lagrangians are stable. 
	\end{conj}
	\begin{rmk}
		This conjecture is not always true, for instance, it does not hold on some non-algebraic K3 surfaces \cite{MW}. However, it is still interesting to ask when the conjecture will hold. 
	\end{rmk}
	Let $Stab$ denote stability manifold of $D^b\mathfrak{Coh}(\mathbb{P}^2)$, the derived category of coherent sheaves of $\mathbb{P}^2$. Bousseau showed that there exists a $1$-dimensional complex submanifold $U\subseteq Stab$ with a diffeomorphism $\phi:U\cong B_0$, to the base of the special Lagrangian of the complement of a smooth cubic in $\mathbb{P}^2$ as affine manifolds \cite{P3}.
	The affine lines of the former are the loci of the form $\mbox{Re}Z^{\sigma}_{\gamma}=c \in \mathbb{R}$ and the affine lines of the later is also in the similar form. Moreover, the charge lattices are isomorphic. Here $Z^{\sigma}_{\gamma}$ is the central charge of some charge $\gamma$ in the charge lattice with respect to the stability condition $\sigma\in Stab$. 
	There are natural scattering diagrams on $U$ and $B_0$: the former is from the counting of semi-stable sheaves of phase zero and the later is the relative Gromov-Witten invariants in Theorem \ref{Gabele}. One interesting result proved by Bousseau is the following:
	\begin{thm} \cite{P3}
		The two scattering diagrams above coincide. 
	\end{thm}
	Recall that $Stab$ admits a natural complex structure and thus $U$ inherits one. On the other hand, $B_0$ is the complement of the discriminant locus of base of the elliptic fibration after hyperK\"ahler rotation. Thus, $B_0$ also admits a natural complex structure. Bousseau showed that $\phi:U\cong B_0$ is actually a biholomorphism with respect the above complex structures \cite[Proposition 2.2.11]{P3}. We may consider the other affine structures on $U$ with affine lines of the form $\mbox{Re}(e^{i\vartheta}Z^{\sigma}_{\gamma})=c\in \mathbb{R}$. We expect that it coincide with the corresponding affine structures of special Lagrangian fibration on $B_0$ of different phase $\vartheta\in S^1$.    	 
	
	It is well-known that the mirror of $\mathbb{P}^2$ is the Landau-Ginzburg superpotential $W=x+y+1/xy: (\mathbb{C}^*)^2\rightarrow \mathbb{C}$. In particular, Auroux-Katzarkov-Orlov proved the homological mirror symmetry $D^b\mathfrak{Coh}(\mathbb{P}^2)\cong FS(W)$, the equivalence between the derived category of coherent sheaves of $\mathbb{P}^2$ with the Fukaya-Seidel category of $W$ \cite{AKO}. Let $\bar{W}:\check{X}\rightarrow \mathbb{C}$ denote the fibrewise compactification of $W$. It is again a Lefschetz fibration and its Fukaya-Seidel category $FS(\bar{W})$ is a deformation of $FS(W)$. Seidel proved that it is actually a trivial deformation, i.e. $FS(\bar{W})\cong FS(W)\otimes \mathbb{C}[[q]]$ \cite{S8}. Therefore, one can construct the stability conditions on $FS(\bar{W})$ via the homological mirror symmetry. From the work of Bousseau, one can identify the base of the Lefschetz fibration as a subset of stability condition of $FS(\bar{W})$. In other words, each fibre of the Lefschetz fibration corresponds to a stability condition of $FS(\bar{W})$.
	
	From the the attractor flow mechanism \cite{KS2} of generalized Donaldson invariants and open Gromov-Witten invariants (the proof of Theorem \ref{116}, see also \cite{L8}) and the identification of initial data will leads to the identification of the generalized Donaldson invariants and the open Gromov-Witten invariants. 
	In particular, the classes supporting semi-stable sheaves with non-trivial countings are exactly corresponding to the relative classes that support special Lagrangian discs with non-trivial counting as holomorphic discs after hyperK\"ahler rotation. This gives some evidence to the Conjecture \ref{stab}.
	\begin{rmk}
		Recently, Graefnitz-Ruddat-Zaslow \cite{GRZ} computed the superpotentials of toric Fano surfaces via the log Gromov-Witten invariants. 
	\end{rmk}

	\subsubsection{Superpotentials of $\mathbb{P}^2$}
	Vianna use the technique of pushing in corners to construct monotone Lagrangians with different superpotentials \cite{V3}. Since the superpotential is invariant under Hamiltonian isotopies, this proved the existence of infinitely many non-Hamiltonian isotopic monotone Lagrangians in $\mathbb{P}^2$. 
	The following is a direct consequence of Theorem \ref{163} and Theorem \ref{169}.
	\begin{thm} Let $Y=\mathbb{P}^2$.
		The superpotential $W(u)$ appear in Theorem \ref{163} are those appeared in the work of Vianna \cite{V3}. In particular, all $n_{\beta}(u)\geq 0$.
	\end{thm} This natural inspires the following conjecture.
	\begin{conj} Given $\omega_i$, $i\gg 1$. 
		There exits a unique monotone special Lagrangian torus fibre in each chamber in $\mathcal{U}_i\subseteq B_0$. 
	\end{conj}


\bibliographystyle{amsplain}

\end{document}